\documentclass[a4paper,10pt,reqno]{amsart}
\usepackage[foot]{amsaddr}
\usepackage{amsmath}
\usepackage{amsthm,enumerate, esint}
\usepackage{mathtools}
\usepackage{hyperref,url}
\usepackage{graphicx}
\usepackage{amssymb}

\usepackage{appendix}

\usepackage[english]{babel}

\usepackage{fancyhdr}

\usepackage{calc}
\usepackage{url}

\usepackage[text={6in,8.6in},centering]{geometry}

\usepackage{srcltx}

\usepackage[T1]{fontenc}

\usepackage[usenames,dvipsnames]{color}

\theoremstyle{plain}
\newtheorem{theorem}{Theorem}[section]
\theoremstyle{plain}
\newtheorem{lemma}{Lemma}[section]
\newtheorem{proposition}{Proposition}[section]

\newtheorem{definition}{Definition}[section]
\newtheorem{remark}{Remark}[section]

\renewcommand{\>}{\right\rangle}

\newcommand{\eps}{\varepsilon}

\newcommand{\To}{\longrightarrow}
\newcommand{\be} {\begin{equation}}
\newcommand{\ee} {\end{equation}}
\newcommand{\bea} {\begin{eqnarray}}
\newcommand{\eea} {\end{eqnarray}}
\newcommand{\Bea} {\begin{eqnarray*}}
	\newcommand{\Eea} {\end{eqnarray*}}

\newcommand{\al} {\alpha}
\newcommand{\ba} {\beta}

\newcommand{\ga} {\gamma}
\newcommand{\Ga} {\Gamma}
\newcommand{\Om} {\Omega}

\newcommand{\De} {\Delta}
\newcommand{\la} {\lambda}

\newcommand{\nequiv} {\not\equiv}

\newcommand{\no} {\nonumber}
\newcommand{\noi} {\noindent}
\newcommand{\lab} {\label}

\newcommand{\R}{\mathbb R}

\newcommand{\Rn}{\mathbb R^N}

\newcommand{\deb}{\rightharpoonup}
\newcommand{\Hs}{\dot{H}^s(\mathbb{R}^{N})}
\newcommand{\Iaf}{I_{a,f}(u)}

\newcommand{\hms}{(\dot{H}^{s})'}
\catcode`\@=11

\newcommand{\authorfootnotes}{\renewcommand\thefootnote{\@fnsymbol\c@footnote}}%
\allowdisplaybreaks

\numberwithin{equation}{section} \allowdisplaybreaks

\begin{document}
        \title[Multiplicity of positive solutions]{On multiplicity of positive solutions\\for nonlocal equations with critical nonlinearity}

\date{}

\author[Mousomi Bhakta]{Mousomi Bhakta\textsuperscript{1}}
\address{\textsuperscript{1}Department of Mathematics, Indian Institute of Science Education and Research, Dr. Homi Bhaba Road, Pune-411008, India}
\email{mousomi@iiserpune.ac.in}

\author[Patrizia Pucci]{Patrizia Pucci\textsuperscript{2}}
\address{\textsuperscript{2}Dipartimento di Matematica e Informatica, Universit\`a degli Studi di Perugia --
Via Vanvitelli 1, I-06123 Perugia, Italy}
\email{patrizia.pucci@unipg.it}


\begin{abstract}
This paper deals with existence and multiplicity of positive solutions to the following class of nonlocal equations with critical nonlinearity:
\begin{equation}
  \tag{$\mathcal E$}
\left\{\begin{aligned}
		(-\Delta)^s u &= a(x) |u|^{2^*_s-2}u+f(x)\;\;\text{in}\;\mathbb{R}^{N},\\
		u &\in \dot{H}^s(\mathbb{R}^{N}),
		 \end{aligned}
  \right.
\end{equation}
where $s \in (0,1)$, $N>2s$, $2_s^*:=\frac{2N}{N-2s}$, $0< a\in L^\infty(\mathbb{R}^{N})$ and $f$ is a nonnegative nontrivial functional in the dual space of $\dot{H}^s$ i.e., $\prescript{}{(\dot{H}^{s})'}{\langle}f,u{\rangle}_{\dot{H}^s}\geq 0$, whenever $u$ is a nonnegative function  in $\dot{H}^s$. We prove existence of a positive solution whose energy is negative. Further, under the additional assumption that $a$ is a continuous function, $a(x)\geq 1$ in $\mathbb{R}^{N}$, $a(x)\to 1$ as $|x|\to\infty$ and $\|f\|_{\dot{H}^s(\mathbb{R}^{N})'}$ is small enough (but $f\not\equiv 0$), we establish existence of at least two positive solutions to ($\mathcal E$).
\medskip

\end{abstract}

\maketitle

\section{Introduction}
	In this article we study existence and multiplicity of positive
solutions to the following fractional elliptic equation in $\Rn$
\begin{equation}
  \tag{$\mathcal E$}\label{MAT1}
\left\{\begin{aligned}
		&(-\Delta)^s u = a(x) |u|^{2^*_s-2}u+f(x)\;\;\text{in}\;\mathbb{R}^{N},\\
		&u >0 \quad\text{in}\quad\mathbb{R}^{N},\qquad
		u \in \dot{H}^{s}{(\mathbb{R}^{N})},
		 \end{aligned}
  \right.
\end{equation}
		 where
		 $s \in (0,1)$ is fixed parameter, $N>2s$,\, $2^*_s:=\frac{2N}{N-2s}$,\, $0< a\in L^\infty(\Rn)$,\, $a(x)\to 1$ as $|x|\to\infty$
		 and $f\not\equiv 0$ is a nonnegative functional in the dual space of $\dot{H}^s(\Rn)$. Here $(-\De)^s$ denotes the  fractional Laplace operator which can be defined for the Schwartz class functions $\mathcal{S}(\Rn)$  as follows
\begin{equation} \label{De-u}
  \left(-\Delta\right)^su(x): = c_{N,s}
\, \text{P.V.} \int_{\Rn}\frac{u(x)-u(y)}{|x-y|^{N+2s}} \, {\rm d}y, \quad c_{N,s}= \frac{4^s\Ga(N/2+ s)}{\pi^{N/2}|\Ga(-s)|}.
\end{equation}
Let
$$\dot{H}^s(R^{N}): =\bigg\{u\in L^{2^*_s}(\R^N) \; : \; \iint_{\mathbb{R}^{2N}}\frac{|u(x)-u(y)|^2}{|x-y|^{N+2s}}\,{\rm d}x\,{\rm d}y<\infty\bigg\},$$
be the homogeneous fractional Sobolev space, endowed 	with the inner product
$\langle\cdot,\cdot\rangle_{\dot{H}^s}$ and corresponding Gagliardo norm
	$$\|u\|_{\dot{H}^{s}(\R{^N})}:=\left( \iint_{\mathbb{R}^{2N}} \frac{|u(x)-u(y)|^2}{|x-y|^{N+2s}}\,{\rm d}x\,{\rm d}y\right)^{1/2}.$$
Clearly, $u\in\dot{H}^s(\Rn)$ implies $u\in L^p_{\text{\scriptsize{\rm loc}}}(\Rn)$ for any $p\in[1,2^*_s]$.	
	\begin{definition}
		The function $u \in \dot{H}^s(\Rn)$ is said to be a positive weak solution of \eqref{MAT1} if $u>0$ in $\Rn$ and for every $\phi\in\dot{H}^s(\Rn)$ we have,
		$$
		\iint_{\R^{2N}}\frac{(u(x) - u(y))(\phi(x)-\phi(y))}{|x-y|^{N+2s}}\,{\rm d}x\,{\rm d}y  = \int_{\R{^N}}a(x)u^{2^*_s-1}\phi \, {\rm d}x+ \prescript{}{(\dot{H}^s)'}{\langle}f,\phi{\rangle}_{\dot{H}^s},
		$$
		where $\prescript{}{(\dot{H}^s)'}{\langle}\cdot,\cdot{\rangle}_{\dot{H}^s}$ denotes the duality bracket between the dual space $\dot{H}^s(\Rn)'$
of $\dot{H}^s(\Rn)$ and $\dot{H}^s(\Rn)$
itself.
	\end{definition}
	
\medskip

Under the stated assumptions
equation \eqref{MAT1} can be considered as a perturbation problem of the homogeneous equation:
		\begin{equation}\label{AT0.8}
		\begin{cases}
		&(-\Delta)^s w = w^{2^*_s-1}\;\;\text{in}\;\mathbb{R}^{N},\\
		&w>0 \;\;\text{in}\;\mathbb{R}^{N},\qquad
		w \in \dot{H}^{s}{(\mathbb{R}^{N})}.
		\end{cases}	
		\end{equation}
		
In the celebrated paper \cite{CLO}  Chen, Li and Ou   proved that \eqref{AT0.8} has a unique positive solution $W$ (up to  translations and dilations). Indeed, any positive solution of~\eqref{AT0.8}  is radially symmetric, with respect to some point $x_0\in\Rn$, strictly decreasing in $r=|x-x_0|$, of class $C^{\infty}(\Rn)$ and so of the explicit parametric form
		\be\lab{9-7-1}
		W(x)= c_{N,s}\bigg(\frac{\la}{\la^2+|x-x_0|^2}\bigg)^\frac{N-2s}{2},
		 \ee
		for some $\la>0$.

	 \medskip
	
The main question in this paper is whether positive solutions can still survive for the perturbed equation \eqref{MAT1}.

When the domain is a bounded subset of $\Rn$, in a pioneering work, Tarantello \cite{T} proved existence of two positive solutions for the following nonhomogeneous problem
\be\lab{1-12-1}-\De u=|u|^\frac{4}{N-2}u+f \mbox{ in }\, \Omega, \quad u=0 \mbox{ on }\, \partial\Omega,\ee where $0\leq f\in H^{-1}(\Om)$ satisfies suitable condition. In \cite{CDPM, M} the authors studied existence of sign changing solutions of \eqref{1-12-1}. In the nonlocal case, when the domain is a bounded subset of $\Rn$, existence of positive solution of $\eqref{MAT1}$ in $\Om$ with Dirichlet boundary condition has been proved in \cite{SZY}. Existence of sign changing solutions of
$$(-\De)^su=|u|^\frac{4s}{N-2s}u+\eps f \mbox{ in }\ \Omega, \quad u=0\ \mbox{ in }\, \Rn\setminus\Om,$$  where $f\geq 0, f\in L^\infty(\Om)$ has been studied in \cite{AT} and existence of two positive solutions have been established in \cite{WZ} when $f$ is a continuous function with compact support in $\Om$.

To the best of our knowledge, so far there has been no papers in the literature, where existence and multiplicity of positive solutions of fractional Laplace equations, with the critical exponents in $\Rn$, have  been established in the non homogeneous case $f(x)\neq 0$. The results in this paper are new even in the local case $s=1$, but we leave the obvious changes, when $s=1$, to the interested reader.
\smallskip

{From now on} we assume that $f$ satisfies the following condition

\begin{enumerate}
\item[${\bf (F)}$] {\it $f\not\equiv 0$ is a nonnegative functional in the dual space $\dot{H}^s(\Rn)'$ of $\dot{H}^s(\Rn)$}.
\end{enumerate}

\noindent Let us state the main results.

\begin{theorem}\lab{th:2}
Assume $0<a\in L^\infty(\Rn)$ and  ${\bf (F)}$ is satisfied. There exists $d>0$ such that if
$\|f\|_{\hms}\leq d$, then equation \eqref{MAT1} admits a positive solution whose energy is negative.
\end{theorem}

\noindent Next, under an additional hypothesis on $a$, we prove existence of at least two positive solutions.

\begin{enumerate}
\item[${\bf (A)}$] {\it $a\in  C(\Rn)\cap L^\infty(\Rn),\quad a(x)\geq 1$ for all $x\in \R^N$,
and $a(x)\rightarrow 1$ as $|x|\rightarrow\infty$.}
\end{enumerate}

\begin{theorem}\lab{th:ex-f}
Assume that  ${\bf (A)}$ and ${\bf (F)}$ are satisfied. If
$$ \|f\|_{\dot{H}^s(\Rn)'}<C_0S^{\tfrac{N}{4s}}, \quad\text{where}\quad C_0:=\bigg(\frac{4s}{N+2s}\bigg)\big((2^*_s-1)\|a\|_{L^\infty(\Rn)}
\big)^{-\frac{N-2s}{4s}},$$
then \eqref{MAT1} admits a positive solution.

In addition, if either $a\equiv 1$ or $\|a\|_{L^\infty(\Rn)}\geq \al(N,s)$, where
$\al(N,s)$ is the second zero of the function
\be\lab{12-8-3}\varphi(t):=\frac{s}{N}t^\frac{N+2s}{2s}-\frac{t^2}{2}+\frac{1}{2^*_s},\ee then \eqref{MAT1} admits at least two positive solutions.
\end{theorem}

\medskip

As in the local case, the Sobolev embedding
$\dot{H}^s(\Rn) \hookrightarrow L^{2^*_s}(\Rn)$ is continuous, but
not compact. Thus the variational functional associated to~\eqref{MAT1} fails to satisfy the Palais-Smale  condition, briefly called $(PS)$
condition.  The lack of compactness becomes clear, when one looks at the special case \eqref{AT0.8}. Solutions of \eqref{AT0.8} are invariant under translation and dilation therefore, there is not compactness. Thus the standard variational technique can not be applied directly. Noncompact variational problems have attracted much attention since the late
seventies. Among them, the Yamabe \cite{Y} and the prescribed scalar curvature problems have played an important role. For those, but also for many related elliptic equations, the loss of compactness is caused by the invariant action of the conformal group, or of one of its subgroups, leading to possible spikes formation. To overcome this difficulty, the a priori knowledge of the energy range where the Palais-Smale condition holds is helpful, and sometimes suffices to construct critical points.

\medskip

Now let us briefly explain the methodology to obtain our results. In Theorem~\ref{th:2}, we establish existence of positive solution as a perturbation of $0$ via Mountain Pass theorem. To prove Theorem~\ref{th:ex-f}, we first do the Palais-Smale decomposition of the functional associated with \eqref{MAT1}. Then we decompose $\dot{H}^s(\Rn)$ into three components which are homeomorphic to the interior, boundary and the exterior of the unit ball in $\dot{H}^s(\Rn)$ respectively.  Thus, using assumption $(A)$, we prove that the energy functional associated to \eqref{MAT1} attains its infimum on one of the components which serves as our first positive solution.  The second positive solution is obtained via a careful analysis on the $(PS)$ sequences associated to the energy functional and we construct a min--max critical level $\ga$, where the $(PS)$ condition holds.  That leads to the existence of second positive solution.

\vspace{2mm}		
			
This paper has been organised in the following way:
In Section 2, we prove the Palais-Smale decomposition theorem associated with the functional corresponding to \eqref{MAT1}. In Section 3, we show existence of two positive solutions of \eqref{MAT1} under the assumption $(A)$, namely
Theorem~\ref{th:ex-f}. In Section~4, we prove Theorem \ref{th:2}. Appendix~A basic properties of the Morrey spaces.

\vspace{2mm}
	
	{\bf Notation:} In this paper $\dot{H}^s(\Rn)'$ (or in short $\hms$) denotes the dual space of $\Hs$, $C$ denotes the generic constant which may vary from line to line. Moreover, $u_+:=\max\{u, 0\}$ and $u_-:=-\min\{u, 0\}$. Therefore, according to our notation $u=u_+-u_-$. Finally, $W$  denotes the unique positive solution of \eqref{AT0.8} and $S$ the best Sobolev constant.

	\section{Palais-Smale  characterization}

In this section we study the Palais-Smale sequences (in short, $(PS)$ sequences) of the functional associated to \eqref{MAT1}.

	\bea\label{EF}
			\bar I_{a,f}(u)&=& \frac{1}{2}\iint_{\R^{2N}}\frac{|u(x) - u(y)|^2}{|x-y|^{N+2s}}\,{\rm d}x\,{\rm d}y  - \frac{1}{2^*_s}\int_{\R^N}a(x)|u|^{2^*_s}\,{\rm d}x-\prescript{}{(\dot{H}^s)'}{\langle}f,u{\rangle}_{\dot{H}^s}\no\\
&=&\frac{1}{2}\|u\|^{2}_{\dot{H}^s(\R^N)}-\frac{1}{2^*_s}\int_{\R^N}a(x)|u|^{2^*_s}\,{\rm d}x -\prescript{}{(\dot{H}^s)'}{\langle}f,u{\rangle}_{\dot{H}^s}.
		\eea

We say that the sequence $(u_k)_k\subset \dot{H}^s(\R^N)$ is a $(PS)$ sequence for $\bar I_{a,f}$ at level $\ba$ if $\bar I_{a,f}(u_k)\to \ba$ and $\bar I_{a,f}'(u_k)\to 0$ in $\big(\dot{H}^s(\Rn)\big)'$. It is easy to see that the weak limit of a $(PS)$ sequence solves \eqref{MAT1} except the positivity.

However the main
difficulty is that the $(PS)$ sequence may not converge strongly and hence the weak limit can be zero even if $\ba>0.$
 The main purpose of this section is to classify $(PS)$ sequences for
the functional $\bar I_{a,f}$. Classification of $(PS)$ sequences has been done for various problems having lack of compactness, to quote a few, we cite \cite{BCG, CF, Lions, PS-1, PS-2, Sm, S}. We establish
 a classification theorem for the $(PS)$ sequences of \eqref{EF} in the spirit of the above results.\\

{\bf Throughout this section we assume $0<a\in L^\infty(\Rn)$, $a(x)\to 1$ as $|x|\to\infty$ and $f$ is a nontrivial element of $\dot{H}^s(\Rn)'$.}		
	
	\begin{proposition}\label{PSP}
				
Let $(u_k)_k \subset \dot{H}^s(\R^N)$ be a $(PS)$ sequence for $\bar I_{a,f}$. Then there exists a subsequence (still denoted by $u_k$) for which
there exist an integer $m\geq 0$, sequences $x_{k}^{j}$, $r_k^j>0$ for $1\leq j \leq m$, functions $\bar{u},\; w_{j}$ for $1\leq j\leq m$ such that
		\be
		(-\De)^s\bar{u} = a(x){|\bar u|}^{2^*_s-2}\bar{u}+f \quad\text{in}\quad\R^N
		\ee
		\be\begin{split}\mbox{either $x_k^j\to x^j\in\Rn$ or }|x_k^j|\to\infty,\quad r_k^j\to 0, \,\, 1\leq j\leq m.\\		 \bigg|\log\big(r^i_kr_k^{-j}\big)\bigg|+\bigg|\frac{x_k^i-x_k^j}{r_k^i}\bigg|
\To\infty\quad \mbox{for } i\neq j,\, \, 1\leq i,\, j\leq m, \end{split}\ee
		\be
		\begin{split}
			(-\De)^s w_j = a(x^j)|w_{j}|^{2^*_s-2}w_{j} \;\mbox{ in }\;\R^N\\
		w_j \nequiv 0,\	w_j \in \dot{H}^s(\R^N),
		\end{split}
		\ee
		\be
		\begin{split}
			u_k-\bigg(\bar{u} +\sum_{j=1}^{m}a(x^j)^{-\frac{N-2s}{4s}}(w_j)^{r_k^j, x_k^j}\big)\bigg) \rightarrow 0 \;\text{as}\;k\rightarrow\infty,\\
			\mbox{where } (w_j)^{r, y}:=r^{-\frac{N-2s}{2}}w_j(\frac{x-y}{r}), \lab{12-13-3}\\
			\bar I_{a,f}(u_k)\rightarrow \bar I_{a,f}(\bar{u})+\sum_{j=1}^{m}a(x^j)^{-\frac{N-2s}{2s}}\bar I_{1,0}(w_j)\;\mbox{ as }\;k\rightarrow \infty,
		\end{split}
		\ee		
		\noi where  in the case $m=0$ the above expressions hold without $w_j$, $x_{k}^{j}$ and $r_k^j$.
In addition, if $u_k\geq 0,$ then $\bar{u}\geq 0$ and $w_j\geq 0$ for all $1\leq j \leq m$. Therefore, $w_j=W$ for all $1\leq j\leq m$ due to the uniqueness up to the translation and dilation for the positive solutions of \eqref{AT0.8}.
				\end{proposition}

\begin{remark}\lab{r:psp}
{\rm From Proposition \ref{PSP}, we see that if $(u_k)_k $ is any nonnegative $(PS)$ sequence for $\bar I_{a,f}$ at level $c$, then $(u_k)_k $ satisfies the $(PS)$ condition  if $c$ can not be decomposed as $c=\bar I_{a,f}(\bar{u})+\sum_{j=1}^ma(x^j)^{-\frac{N-2s}{2}}\bar I_{1,0}(W)$, where $m\geq 1$ and $W$ is the unique positive radial solution of \eqref{AT0.8}.}
\end{remark}

\medskip

Before starting the proof of this proposition, we prove an auxiliary lemma

\begin{lemma}\label{L2}
			Let $(\phi_k)_k$ weakly converge to $\phi$ in $\Hs$ and a.e. in $\R^N$, then
			$$a|\phi_k|^{2^*_s-2}\phi_k - a|\phi|^{2^*_s-2}\phi \longrightarrow 0 \quad \text{in}\quad \dot{H}^s(\R^N)'.$$
		\end{lemma}
		\begin{proof}
			Defining $\psi_k$ as  $\phi_k - \phi$, we see $\psi_k \rightharpoonup 0$ in $\Hs$. In particular, $(\psi_k)_k$ is bounded in $\Hs$. Thus, up to a subsequence, $\psi_k \to 0$ in $L_{\rm loc}^{q}(\R^N) \ \mbox{for all } \ 1<q<2_{s}^*$ and $\psi_k \to 0$ a.e.
in~$\R^N$. Consequently,
			$a|\phi +\psi_k|^{2^*_s-2}(\phi + \psi_k) - a|\phi|^{2^*_s-2}\phi \rightarrow 0$ a.e..  We also observe that
			for every $\varepsilon>0, \;\text{there exists }C_\varepsilon >0$ such that
			\be\lab{10-7-2}\bigg|a|\phi + \psi_k|^{2^*_s-2}(\phi + \psi_k) - a|\phi|^{2^*_s-2}\phi \bigg|^{\tfrac{2N}{N+2s}} \leq \varepsilon |\psi_k|^{2^*_s} + C_{\varepsilon} |\phi|^{2^*_s} .\ee
Moreover, since $\psi_k \rightharpoonup 0$ in $\Hs$ implies $(\psi_k)_k$ is uniformly bounded in $L^{2^*_s}(\Rn)$ and the fact that $|\phi|^{2^*_s}\in L^1(\Rn)$,  using Vitaly's convergence theorem, it is easy to see from \eqref{10-7-2} that
$$a|\phi +\psi_k|^{2^*_s-2}(\phi + \psi_k) - a|\phi|^{2^*_s-2}\phi   \rightarrow 0\quad\mbox{in }L^{\tfrac{2N }{N+2s}}_{\rm loc}(\Rn).$$ Moreover, using \eqref{10-7-2}, we also see that given any $\eps>0$, there exists $R>0$ such that
\be\lab{10-7-3}
\int_{\Rn\setminus B(0,R)}\bigg|a|\phi + \psi_k|^{2^*_s-2}(\phi + \psi_k) - a|\phi|^{2^*_s-2}\phi \bigg|^{\tfrac{2N}{N+2s}}\, {\rm d}x<\eps.
\ee
As a result, $a|\phi +\psi_k|^{2^*_s-2}(\phi + \psi_k) - a|\phi|^{2^*_s-2}\phi \rightarrow 0$ in $L^{\tfrac{2N}{N+2s}}(\Rn)$. Since $\dot{H}^s(\Rn)$ is continuously embedded in $L^{2^*_s}(\Rn)$, which is the dual space of $L^{\tfrac{2N}{N+2s}}(\Rn)$, it follows that $a|\phi +\psi_k|^{2^*_s-2}(\phi + \psi_k) - a|\phi|^{2^*_s-2}\phi \rightarrow 0$ in $\dot{H}^{s}(\Rn)'$.
\end{proof}	

{\bf Proof of Proposition \ref{PSP}:}

\begin{proof}
      We divide the proof into few steps.\\

      \underline{\bf Step 1:} Using standard arguments it follows that $(PS)$ sequences for $\bar I_{a,f}$ are bounded in $\Hs$. More precisely, as $k\to\infty$
      \begin{eqnarray*}
   \lim_{k\to\infty} \bar I_{a,f}(u_k) +o(1)+o(1)\|u_k\|_{H^s(\Rn)}&\geq& \bar I_{a,f}(u_k) \, - \,
      \frac{1}{2_s^*} \prescript{}{\hms}{\langle}\bar I'_{a,f}(u_k), u_k{\rangle}_{\dot{H}^s} \\
      & =&\bigg(\frac{1}{2}-\frac{1}{2_s^*}\bigg)\|u_k\|_{\Hs}^{2}
       - \left(1- \frac{1}{2_s^*} \right) \prescript{}{\hms}{\langle}f, u_k{\rangle}_{\dot{H}^s}\\
       &\geq&\bigg(\frac{1}{2}-\frac{1}{2_s^*}\bigg)\|u_k\|_{\Hs}^{2}-  \left(1- \frac{1}{2_s^*} \right)\|f\|_{\hms}\|u_k\|_{\dot{H}^s}.
      \end{eqnarray*}
    This immediately implies $(u_k)_k$ is bounded in $\Hs$. Consequently, up to a subsequence $ u_k \rightharpoonup \bar{u}$ in $\Hs$.  Moreover, as $\prescript{}{\hms}{\langle}\bar I_{a,f}'(u_k), v{\rangle}_{\dot{H}^s}\rightarrow 0 $
      as $k\rightarrow\infty$ for all  $v\in\Hs$, we have
		\be\label{B6}
		(-\De)^su_k - a(x)|u_k|^{2^*_s-2}u_k -f\underset{k}{\rightarrow} 0\quad \text{in}\quad \dot{H}^s(\R^N)'.
		\ee

	\medskip		
				
		\underline{\bf Step 2:}
		From \eqref{B6} we get by letting $k \rightarrow \infty$
		\be\lab{8-1-1}\iint_{\R^{2N}} \frac{(u_k(x) - u_k(y))((v(x) - v(y))}{|x-y|^{N+2s}} \, {\rm d}x\, {\rm d}y \,  - \, \int_{\R^N}a(x)\, |u_k|^{2^*_s-2} u_k v \,
		{\rm d} x \, - \, \prescript{}{\hms}{\langle}f, v{\rangle}_{\dot{H}^s} \, {\rightarrow} \,  0,
		\ee
	for all $v\in\Hs$. Moreover,  $u_k\rightharpoonup \bar u$ in $\dot{H}^s(\Rn)$ implies that
$$\iint_{\R^{2N}} \frac{(u_k(x) - u_k(y))((v(x) - v(y))}{|x-y|^{N+2s}} \, {\rm d}x\, {\rm d}y \longrightarrow \, \iint_{\R^{2N}} \frac{(\bar u(x) - \bar u(y))((v(x) - v(y))}{|x-y|^{N+2s}} \, {\rm d}x\, {\rm d}y.$$
Furthermore, using Lemma~\ref{L2} we conclude
 $$\int_{\R^N}a(x)\, |u_k|^{2^*_s-2} u_k v \,{\rm d} x \,   \longrightarrow \, \int_{\R^N}a(x)\, |\bar u|^{2^*_s-2} \bar u v \, {\rm d} x. $$
 Therefore, passing the limit in \eqref{8-1-1}, we have
 $$(-\De)^s\bar u = a(x) \, |\bar u|^{2^*_s-2}\bar u\, + \, f\text{ in }\Rn,\quad \bar u\in \Hs. $$

\medskip

  \underline{\bf Step 3:} In this step we show that $(u_k-\bar u)_k$ is a $(PS)$ sequence for $\bar I_{a,0}$ at the level \\
$\lim_{k\to\infty} \bar I_{a,f}(u_k)-\bar I_{a,f}(\bar u)$ and $u_k-\bar u\rightharpoonup 0$ in $\Hs$.

\medskip

To see this, first we observe that as $k\to\infty$
$$\|u_k-\bar{u}\|_{\Hs}^2 = \|u_k\|_{\Hs}^2-\|\bar{u}\|_{\Hs}^2+o(1)$$
and by
the Br\'ezis-Lieb lemma
$$\int_{\Rn}a(x)|u_k-\bar{u}|^{2_s^*}{\rm d}x = \int_{\Rn}a(x)|u_k|^{2_s^*}{\rm d}x - \int_{\Rn}a(x)|\bar{u}|^{2_s^*}{\rm d}x + o(1).
$$
    Further as $u_k\rightharpoonup u$ and $f\in \dot{H}^s(\Rn)'$, we also have
\be\lab{17-8-4}
\prescript{}{\hms}{\langle}f,u_k{\rangle}_{\dot{H}^s}\To \prescript{}{\hms}{\langle}f,u{\rangle}_{\dot{H}^s}.
\ee
Using above, it follows that
     \Bea
     \bar I_{a,0}(u_k-\bar{u})
     &=& \frac{1}{2}\|u_k\|_{\Hs}^2 -\frac{1}{2_s^*}\int_{\Rn}a(x)|u_k|^{2_s^*}{\rm d}x-\prescript{}{\hms}{\langle}f, u_k{\rangle}_{\dot{H}^s}\no\\
     &\;&-\left\{\frac{1}{2}\|\bar{u}\|_{\Hs}^2 -\frac{1}{2_s^*}\int_{\Rn}a(x)|\bar{u}|^{2_s^*}{\rm d}x-\prescript{}{\hms}{\langle}f, \bar{u}{\rangle}_{\dot{H}^s}\right\}+o(1)\no\\
     &\To& \lim_{k\to\infty}\bar I_{a,f}(u_k)-\bar I_{a,f}(\bar{u}).
     \Eea
     As $\prescript{}{\hms}{\langle}\bar I'_{a,f}(\bar u), v{\rangle}_{\dot{H}^s}=0$
     for any $v\in\Hs$, we obtain
     \bea
     \prescript{}{\hms}{\langle}\bar I'_{a,0}(u_k-\bar u),\, v{\rangle}_{\dot{H}^s} &=& \langle u_k-\bar{u},v\rangle_{\dot{H}^s} -\int_{\Rn}a(x)|u_k-\bar{u}|^{2_s^*-2}(u_k-\bar{u})v\, {\rm d}x\no\\
     &=& \langle u_k,v\rangle_{H^s} -\int_{\R{^N}}a(x)|u_k|^{2_s^*-2}u_k v-\prescript{}{\dot{H}^{-s}}{\langle}f, v{\rangle}_{\dot{H}^s}\no\\
     &\;& -\langle \bar{u},v\rangle_{H^s} +\int_{\R{^N}}a(x)|\bar{u}|^{2_s^*-2}\bar{u} v+\prescript{}{\dot{H}^{-s}}{\langle}f, v{\rangle}_{\dot{H}^s}\no\\
     &\;&+\int_{\R{^N}}a(x)\bigg\{|u_k|^{2_s^*-2}u_k-|\bar{u}|^{2_s^*-2}\bar{u}-|u_k-\bar{u}|^{2_s^*-2}(u_k-\bar{u})\bigg\}v\, {\rm d}x\no\\
     &=& o(1)+\int_{\R{^N}}a(x)\bigg\{|u_k|^{2_s^*-2}u_k-|\bar{u}|^{2_s^*-2}\bar{u}-|u_k-\bar{u}|^{2_s^*-2}(u_k-\bar{u})\bigg\}v\, {\rm d}x\no
     \eea
     \textbf{\underline{Claim }:} $\displaystyle\int_{\Rn}a(x)\left\{|u_k|^{2_s^*-2}u_k-|\bar{u}|^{2_s^*-2}\bar{u}-|u_k-\bar{u}|^{2_s^*-2}(u_k-\bar{u})\right\}v {\rm d}x =o(1),\;\forall\, v\in\Hs.$\\

 To prove the claim,  we note that
 \bea\lab{12-10-2}
 \bigg |a\left\{|u_k|^{2_s^*-2}u_k-|\bar{u}|^{2_s^*-2}\bar{u}-|u_k-\bar{u}|^{2_s^*-2}(u_k-\bar{u})\right\}\bigg| \leq C\bigg(|u_k-\bar u|^{2^*_s-2}|\bar u|+|u|^{2^*_s-2}|u_k-\bar u|\bigg)
 \eea
since  $u_k\rightharpoonup \bar u$ in $\Hs$ implies $(u_k-\bar u)_k$ is uniformly bounded in $\Hs$ and thus also bounded in $L^{2_s^*}(\Rn)$.
Moreover, as $|\bar{u}|^{2_s^*}\in L^{1}(\Rn)$, using H\"{o}lder inequality on the RHS of \eqref{12-10-2}, given $\eps>0$ there exists $R=R(\eps)>0$ such that
\bea\lab{12-10-3}
&&\bigg |\int_{\Rn\setminus B(0,R)} a\left\{|u_k|^{2_s^*-2}u_k-|\bar{u}|^{2_s^*-2}\bar{u}-|u_k-\bar{u}|^{2_s^*-2}(u_k-\bar{u})\right\}v\,dx\bigg| \no\\
&\leq&C\bigg(\int_{\Rn}|u_k-\bar u|^{2^*_s}dx\bigg)^\frac{2^*_s-2}{2^*_s}\bigg(\int_{\Rn\setminus B(0,R)}|\bar u|^{2^*_s}dx\bigg)^\frac{1}{2^*_s}\bigg(\int_{\Rn\setminus B(0,R)}|v|^{2^*_s}dx\bigg)^\frac{1}{2^*_s}\no\\
&&+C\bigg(\int_{\Rn\setminus B(0,R)}|\bar u|^{2^*_s}dx\bigg)^\frac{2^*_s-2}{2^*_s}\bigg(\int_{\Rn}|u_k-\bar u|^{2^*_s}dx\bigg)^\frac{1}{2^*_s}\bigg(\int_{\Rn\setminus B(0,R)}|v|^{2^*_s}dx\bigg)^\frac{1}{2^*_s}\no\\
&<&\eps.
\eea
Similarly using \eqref{12-10-2} and Vitaly's convergence theorem, we also obtain $$\bigg |\int_{ B(0,R)} a\left\{|u_k|^{2_s^*-2}u_k-|\bar{u}|^{2_s^*-2}\bar{u}-|u_k-\bar{u}|^{2_s^*-2}(u_k-\bar{u})\right\}v\,dx\bigg| =o(1).$$
Combining this with \eqref{12-10-3}, the claim follows and hence Step 3 follows.

  \medskip

\underline{\bf Step 4:}  Rescaling of $(v_k)_k$ in the nontrivial case.

If $u_k\to \bar u \mbox{ in } \Hs$, then the theorem is proved with $m=0$. Therefore, we assume $u_k\not\to \bar u \mbox{ in } \Hs$.
Set  $$v_k:=u_k-\bar{u}.$$
From Step 3, we have $(-\De)^sv_k - a(x)|v_k|^{2_s^*-2}v_k \to 0\text{ in }\, \dot{H}^s(\Rn)'.$ Therefore, as $(v_k)_k$ is uniformly bounded in $\Hs$, $$\|v_k\|_{\Hs}^2= \int_{\R{^N}}a(x)|v(x)|^{2_s^*}{\rm d}x\leq\|a\|_{L^\infty(\Rn)} \|v_k\|_{L^{2^*_s}(\Rn)}^{2^*_s}.$$
Consequently, $v_k\not\To 0$ in $L^{2_s^*}(\Rn)$ and, up to a subsequence,
\be\lab{12-10-4} \inf_{k}\|v_k\|_{L^{2^*_s}(\Rn)}\geq C>0.\ee
Moreover, since $(v_k)_k$ is bounded in $\Hs$ and $\Hs\xhookrightarrow{}L^{2_s^*}(\Rn)\xhookrightarrow{}\mathcal{L}^{2,N-2s}(\Rn)$ (See Appendix A), we have
$\|v_k\|_{\mathcal{L}^{2,N-2s}(\Rn)}\leq c$ for some $c>0$ (independent of $k$). On the other hand, combining \eqref{12-10-4} with Lemma \ref{l:12-13-1} for $r=2$, we readily see that $\|v_k\|_{\mathcal{L}^{2,N-2s}(\Rn)}\geq \tilde C$, for some $\tilde C>0$ independent of $k$. Hence, there exists a positive constant, which we denote by $C$ again such that, for all $k$
\be\lab{12-10-5}
C\leq \|v_k\|_{\mathcal{L}^{2,N-2s}(\Rn)}\leq C^{-1}.
\ee
Combining \eqref{12-10-5} with the definition of $\mathcal{L}^{2,N-2s}(\Rn)$, we deduce that for every $k\in\mathbb{N}$, there exists $r_k>0$,  $x_k\in\Rn$ such that
\be\lab{12-10-6}r_k^{-2s}\int_{B(x_k,r_k)}|v_k|^2\,dx=r_k^{N-2s} \fint_{B(x_k,r_k)}|v_k|^2dx\geq \|v_k\|_{\mathcal{L}^{2,N-2s}(\Rn)}-\frac{C^2}{2k}\geq \bar C,\ee
for some $\bar C>0$ (independent of $k$).

Now we define, $\tilde{v}_k := r_k^{\frac{N-2s}{2}}v_k(r_kx+x_k).$
In the view of the scaling invariance of the $\Hs$ norm, $(\tilde v_k)_k$ is a bounded sequence in $\Hs$, thus up to a subsequence
$\tilde{v}_k\rightharpoonup \tilde{v}$ in $\Hs$. Consequently, $\tilde v_k\to \tilde v$ in $L^2_{\rm loc}(\Rn)$. Therefore, using change of variable, we observe from \eqref{12-10-6}
$$0<r_k^{-2s}\int_{B(x_k,r_k)}|v_k|^2\,dx=\int_{B(0,1)}|\tilde v_k|^2dx\To\int_{B(0,1)}|\tilde v|^2dx.$$
Hence $\tilde v\neq 0$. Clearly, up to a subsequence, either $x_k\to x_0\in\Rn$ or $|x_k|\to\infty$. Also note that  $\tilde v_k\rightharpoonup \tilde v\neq 0$ and $v_n\rightharpoonup 0$ implies $r_k\to 0$.

\medskip

\underline{\bf Step 5:} In this step we prove that $\tilde v$ solves
$$(-\De)^s \tilde v = a(x_0)|\tilde v|^{2^*_s-2}\tilde v\, \mbox{ in }\, \R^N, \quad
			\tilde v\in \dot{H}^s(\R^N),$$ or equivalently $a(x_0)^\frac{N-2s}{4s}\tilde{v}$ solves \eqref{AT0.8},
without the sign restriction.
			
To this aim, it is enough to show that for arbitrarily chosen $\varphi\in C^\infty_c(\Rn)$ the following holds:
			$$\langle\tilde v,\, \varphi \rangle_{\dot{H}^s}=\int_{\Rn}a(x_0)|\tilde v|^{2^*_s-2}\tilde v\varphi.$$
Let $\varphi\in C^\infty_c(\Rn)$ be arbitrary. By Step 3, we have $\bar I'_{a,0}(v_k)\to 0$ in $\dot{H}^s(\Rn)'$. Therefore,  as $\tilde v_k\rightharpoonup \tilde v$, using change of variables, we get
\begin{align*}
	\langle\tilde v,\, \varphi \rangle_{\dot{H}^s} &=\lim_{k\to\infty}\langle\tilde v_k,\, \varphi \rangle_{\dot{H}^s}\no\\
&=\lim_{k\to\infty}\iint_{\R^{2N}}\frac{r_k^\frac{N-2s}{2}\big(v_k(r_kx+x_k)-v_k(r_ky+x_k)\big)
\big(\varphi(x)-\varphi(y)\big)}{|x-y|^{N+2s}}dxdy\no\\
&=\lim_{k\to\infty}\iint_{\R^{2N}}\frac{r_k^{-\frac{N-2s}{2}}\big(v_k(x)-v_k(y)\big)\big(\varphi\big(\dfrac{x-x_k}{r_k}\big)
-\varphi\big(\dfrac{y-x_k}{r_k}\big)\big)}{|x-y|^{N+2s}}dxdy\no\\
&=\lim_{k\to\infty}\int_{\R^{N}}a(x)|v_k(x)|^{2^*_s-2}v_k(x)r_k^{-\frac{N-2s}{2}}\varphi\big(\dfrac{x-x_k}{r_k}\big)dx\no\\
&=\lim_{k\to\infty}\int_{\R{^N}}a(r_kx+x_k)|\tilde{v}_k(x)|^{2_s^*-2}\tilde{v}_k(x)\varphi(x)\, dx.
\end{align*}		
{\bf Claim:} $\lim_{k\to\infty}\displaystyle\int_{\R{^N}}a(r_kx+x_k)|\tilde{v}_k(x)|^{2_s^*-2}\tilde{v}_k(x)\varphi(x)\, dx=a(x_0)\int_{\Rn}|\tilde v|^{2^*_s-2}\tilde v\varphi\, dx.$	
		
To see this 			
\Bea
&&\bigg|\int_{\Rn}a(r_kx+x_k)|\tilde{v}_k(x)|^{2_s^*-2}\tilde{v}_k(x)\varphi(x)\, dx-a(x_0)\int_{\Rn}|\tilde v|^{2^*_s-2}\tilde v\varphi\, dx \bigg |\\
&\leq&\bigg|\int_{\Rn}a(r_kx+x_k)\big(|\tilde{v}_k|^{2_s^*-2}\tilde{v}_k-|\tilde v|^{2^*_s-2}\tilde v\big)\varphi\bigg| + \bigg|\int_{\Rn}\big(a(r_kx+x_k)-a(x_0)\big)|\tilde v|^{2^*_s-2}\tilde v\varphi\, dx \bigg |\\
&=&I_k+J_k.
\Eea			
Since $r_k\to 0$, $x_k\to x_0$, $a\in C(\Rn)$ and $|\tilde v|^{2^*_s-2}\tilde v\varphi\in L^1(\Rn)$
by the H\"{o}lder inequality,
the dominated convergence theorem gives that $\lim_{k\to\infty}J_k=0$. On the other hand, as $\varphi$ has compact support and $\tilde v_k\to \tilde v$ a.e. by Vitaly's convergence theorem, it is not difficult to see that $\lim_{k\to\infty}I_k=0$. Thus the claim follows. Hence, Step 5 is proved. Equivalently $a(x_0)^\frac{N-2s}{4s}\tilde{v}$ solves \eqref{AT0.8}, without sign requirement.

\medskip

Define, $$z_k(x) := v_k(x)-r_k^{-\frac{N-2s}{2}}\tilde{v}\left(\frac{x-x_k}{r_k}\right).$$

\medskip

\underline{\bf Step 6:}
In this step we show that $(z_k)_k$ is a $(PS)$ sequence for $\bar I_{a,0}$ at the level $\lim_{k\to\infty}\bar I_{a,f}(u_k)-\bar I_{a,f}(\bar{u})-a(x_0)^{-\frac{N-2s}{2s}}\bar I_{1,0}
(w),$ where $w$ is a solution of \eqref{AT0.8}, without the sign condition.

\vspace{2mm}

To see this, first we observe that if we define, $\tilde z_k:= r_k^\frac{N-2s}{2}z_k(r_kx+x_k)$, then it is easy to check that  $\tilde z_k= \tilde{v}_k-\tilde{v}$. Therefore, the scaling invariance in the norm of $\Hs$ gives
\be\lab{12-11-1}\|z_k\|_{\Hs}=\|\tilde z_k\|_{\Hs}=\|\tilde{v}_k-\tilde{v}\|_{\Hs}.\ee

{\bf Claim 1:} \be\lab{12-11-3}\displaystyle\int_{\R{^N}}a(r_kx+x_k)|\tilde{v}_k(x)-\tilde v(x)|^{2_s^*}{\rm d}x=\int_{\Rn}a(r_kx+x_k)|\tilde{v}_k(x)|^{2_s^*}{\rm d}x-\int_{\R{^N}}a(x_0)|\tilde{v}(x)|^{2_s^*}{\rm d}x+o(1).\ee

To prove the claim, we set $a_k:=a(r_kx+x_k)$. An elementary analysis yields for any $p>1$,
\be\lab{12-11-5}
\bigg|\tilde v_k|^{p-1}\tilde v_k-|\tilde v|^{p-1}\tilde v-|\tilde v_k-\tilde v|^{p-1} (\tilde v_k-\tilde v)\bigg| \leq
C\bigg(|\tilde v_k-\tilde v|\big)^{p-1}|\tilde v|+|\tilde v_k-\tilde v| |\tilde v|^{p-1}\bigg)
\ee
Thus,

\bea\lab{12-11-4}
\bigg||a_k^\frac{1}{2^*_s}\tilde v_k|^{2^*_s}-|a_k^\frac{1}{2^*_s}\tilde v|^{2^*_s}-|a_k^\frac{1}{2^*_s}(\tilde v_k-\tilde v)|^{2^*_s} \bigg| &\leq&
C\bigg(\big(a_k^\frac{1}{2^*_s}|\tilde v_k-\tilde v|\big)^{2^*_s-1}|a_k^\frac{1}{2^*_s}\tilde v|+\big|a_k^\frac{1}{2^*_s}(\tilde v_k-\tilde v)\big| |a_k^\frac{1}{2^*_s}\tilde v|^{2^*_s-1}\bigg)\no\\
&\leq& C\|a\|_{L^\infty(\Rn)}\bigg(|\tilde v_k-\tilde v|^{2^*_s-1}|\tilde v|+
|\tilde v_k-\tilde v||\tilde v|^{2^*_s-1}\bigg)
\eea
Using the dominated convergence theorem, we immediately have $\displaystyle{\lim_{k\to\infty}\int_{\Rn}|a_k^\frac{1}{2^*_s}\tilde v|^{2^*_s}dx=\int_{\Rn}a(x_0)|\tilde v|^{2^*_s}dx}$. Therefore, to prove the claim, it is enough to show that
$$(i)\,\, \int_{\Rn}|\tilde v_k-\tilde v|^{2^*_s-1}|\tilde v|dx=o(1)\quad\text{and}\quad (ii)\,\,\int_{\Rn}|\tilde v_k-\tilde v||\tilde v|^{2^*_s-1}dx=o(1).$$
For this, given any $\eps>0$ there exists $R=R(\eps)>0$ such that
$$\int_{\Rn\setminus B(0,R)}|\tilde v_k-\tilde v|^{2^*_s-1}|\tilde v|dx\leq \bigg(\int_{\Rn}|\tilde v_k-\tilde v|^{2^*_s}dx\bigg)^\frac{2^*_s-1}{2^*_s}\bigg(\int_{\Rn\setminus B(0,R)}|\tilde v|^{2^*_s}dx\bigg)^\frac{1}{2^*_s}<\eps,$$
since $\tilde{v}_k\rightharpoonup \tilde{v}$ in $\Hs$ implies that
$(\tilde v_k-\tilde v)_k$ is uniformly bounded in $L^{2^*_s}(\Rn)$. Similarly using Vitaly's convergence theorem via
the H\"{o}lder inequality, it can be also shown that $\displaystyle\int_{B(0,R)}|\tilde v_k-\tilde v|^{2^*_s-1}|\tilde v|dx=o(1)$.  Thus (i) holds. Similarly (ii) can also be proved. Hence the claim follows.

Applying \eqref{12-11-1} and \eqref{12-11-3}, we have
 \begin{align*}
   \bar I_{a,0}(z_k) &=\frac{1}{2}\|z_k\|_{\Hs}^2-\frac{1}{2_s^*}\int_{\R{^N}}a(x)|z_k(x)|^{2_s^*}{\rm d}x\no\\ &=\frac{1}{2}\|\tilde{v}_k-\tilde{v}\|_{\Hs}^2-\frac{1}{2_s^*}\int_{\R{^N}}a(r_kx+x_k)|z_n(r_nx+y_n)|^{2_s^*}r_n^{N}{\rm d}x\no\\
    &=\frac{1}{2}\|\tilde{v}_k-\tilde{v}\|_{\Hs}^2-\frac{1}{2_s^*}\int_{\R{^N}}a(r_nx+y_n)|(\tilde{v}_k-\tilde{v})(x)|^{2_s^*}{\rm d}x\no\\
    &= \frac{1}{2}\left\{\|\tilde{v}_k\|_{\Hs}^2-\|\tilde{v}\|_{\Hs}^2\right\}-\frac{1}{2_s^*}\int_{\R{^N}}a(r_kx+x_k)|\tilde{v}_k(x)|^{2_s^*}{\rm d}x\no\\
    &\qquad\qquad\qquad\qquad\qquad\qquad+\frac{1}{2_s^*}\int_{\R{^N}}a(x_0)|\tilde{v}(x)|^{2_s^*}{\rm d}x+o(1)\no\\
    &=\frac{1}{2}\|v_k\|_{\Hs}^2-\frac{1}{2_s^*}\int_{\R{^N}}a(x)|v_k|^{2_s^*}{\rm d}x\no\\
    &\qquad-a(x_0)^{-\tfrac{N-2s}{2s}}\left\{\frac{a(x_0)^{ \frac{N-2s}{2s}}}{2}\|\tilde{v}\|_{\Hs}^2-\frac{a(x_0)^\frac{N}{2s}}{2^*_s}\int_{\R{^N}}|\tilde{v}|^{2_s^*}\;{\rm d}x\right\}+o(1)\no\\
    &=\bar I_{a,0}(v_k)-a(x_0)^{-\tfrac{N-2s}{2s}}\bar I_{1,0}\big(a(x_0)^\frac{N-2s}{4s}\tilde v\big)+o(1)\no\\
    &=\lim_{k\to\infty}\bar I_{a,f}(u_k) -\bar I_{a,f}(\bar{u})-a(x_0)^{-\frac{N-2s}{2s}}\bar I_{1,0}(w),\no
    \end{align*}
    where $w$ is a solution of \eqref{AT0.8} without the sign condition. From the above energy estimate of $z_k$, we also observe that
    $$\bar I_{a,0}(z_k) =\bar I_{a,0}(v_k)-a(x_0)^{-\tfrac{N-2s}{2s}}\bar I_{1,0}(w)\leq \bar I_{a,0}(v_k)-a(x_0)^{-\tfrac{N-2s}{2s}}\bar I_{1,0}(W),$$
where $W$ is the unique positive solution of \eqref{AT0.8}, which also has the minimum energy among all the solutions of \eqref{AT0.8} with or without the sign condition. Further as   $\bar I_{1,0}(W)=\frac{s}{N}S^\frac{N}{2s}$ (see \eqref{17-7-5} and the comments below to it) and $a>0$, we obtain $\bar I_{a,0}(z_k)<\bar I_{a,0}(v_k)$.

 \medskip

Next, we estimate $\prescript{}{\hms}{\langle}I'_{a,0}(z_k), \varphi{\rangle}_{\dot{H}^s}$ for any arbitrarily chosen $\varphi\in C^\infty_c(\Rn)$. Towards this, first we observe that an easy computation yields $\langle z_k,\varphi\rangle_{\Hs}=\langle \tilde z_k,\varphi_k\rangle_{\Hs}$, where \\ $\varphi_k(x):=r_k^\frac{N-2s}{2}\varphi(r_kx+x_k)$.
Clearly, $\|\varphi_k\|_{\Hs}=\|\varphi\|_{\Hs}$ and $\varphi_k\rightharpoonup 0$ in $\Hs$ as $r_k\to 0$.
Using these and the fact that $\tilde z_k=\tilde v_k-\tilde v$, we obtain

\begin{align}\label{12-11-7}
    \prescript{}{\hms}\langle\bar I'_{a,0}(z_k), \varphi\rangle_{\dot{H}^s}&=\langle z_k,\varphi\rangle_{\Hs}- \int_{\Rn}a(x)|z_k|^{2^*_s-2}z_k\varphi\, dx\no\\
    &=\langle \tilde z_k,\varphi_k\rangle_{\Hs}- \int_{\Rn}a(r_kx+x_k)|z_k(r_kx+x_k)|^{2^*_s-2}z_k(r_kx+x_k)\varphi(r_kx+x_k)\, r_k^N \, dx\no\\
    &=\langle \tilde{v}_k-\tilde{v},\varphi_k\rangle_{\Hs}-\int_{\R^N}a(r_kx+x_k)|\tilde z_k|^{2_s^*-2}\tilde z_k\varphi_k\;{\rm d}x\\
    &=\langle \tilde{v}_k,\varphi_k\rangle_{\Hs}-\langle \tilde{v},\varphi_k\rangle_{\Hs}-\int_{\R^N}a(r_kx+x_k)|\tilde z_k|^{2_s^*-2}\tilde z_k\varphi_k\;{\rm d}x\no\\
    &=\langle v_k,\varphi\rangle_{\Hs}+o(1)-\int_{\R^N}a_k(x)|(\tilde v_k-\tilde v)(x)|^{2_s^*-2}(\tilde v_k-\tilde v)(x)\varphi_k\;{\rm d}x,\no
\end{align}
where in the last line we have used the fact that $\varphi_k\rightharpoonup 0$ in $\Hs$ and $a_k(x)=a(r_kx+x_k)$. Now, using \eqref{12-11-5} with $p=2^*_s-1$ and the following an argument similar to the proof of Claim 1, it can be shown that
\be\lab{12-11-6}
 \int_{\R^N}a_k|\tilde{v}_k-\tilde{v}|^{2_s^*-2}(\tilde{v}_k-\tilde{v})\varphi_k\;{\rm d}x = \int_{\R^N}a_k|\tilde{v}_k|^{2_s^*-2}\tilde{v}_k\varphi_k\;{\rm d}x - \int_{\R^N}a_k|\tilde{v}|^{2_s^*-2}\tilde{v}\varphi_k\;{\rm d}x +o(1),
\ee
Since $\|a_k\|_{L^\infty(\Rn)}=\|a\|_{L^\infty(\Rn)}$, $\|\varphi_k\|_{L^{2^*_s}(\Rn)}=\|\varphi\|_{L^{2^*_s}(\Rn)}$ and $\varphi_k\to 0$ a.e., it is easy to see that $\lim_{k\to\infty}\int_{\R^N}a_k|\tilde{v}|^{2_s^*-2}\tilde{v}\varphi_k\;{\rm d}x=0$. On the other hand, using change of variable it follows that
 $$ \int_{\R^N}a_k|\tilde{v}_k|^{2_s^*-2}\tilde{v}_k\varphi_k\;{\rm d}x= \int_{\R^N}a(x)|v_k(x)|^{2_s^*-2}v_k(x)\varphi(x)dx.$$
Substituting these into \eqref{12-11-6}, we obtain
\be\lab{12-11-8}
 \int_{\R^N}a_k|\tilde{v}_k-\tilde{v}|^{2_s^*-2}(\tilde{v}_k-\tilde{v})\varphi_k\;{\rm d}x = \int_{\R^N}a(x)|v_k(x)|^{2_s^*-2}v_k(x)\varphi(x)dx+o(1).
\ee

Thus substituting \eqref{12-11-8} into \eqref{12-11-7} yields
$$    \prescript{}{\hms}\langle I'_{a,0}(z_k), \varphi\rangle_{\dot{H}^s}=\langle v_k,\varphi\rangle_{\Hs}-\int_{\R^N}a(x)|v_k(x)|^{2_s^*-2}v_k(x)\varphi(x)dx+o(1)=\bar I'_{a, 0}(v_k)(\varphi)=0,$$
where for the last equality we have used Step 3. This completes the proof of Step 6.

Now, starting from a $(PS)$ sequence $(v_k)_k$ for $I_{a,0}$ we have extracted another $(PS)$ sequence $(z_k)_k$ at a level which is strictly lower than the previous one, with a fixed minimum amount of decrease. Since, $\sup_k\|v_k\|_{\Hs}\leq C$ (finite), hence the process should terminate after finitely many steps and the last $(PS)$ sequence strongly converges to $0$. Further, $\bigg|\log\big(\frac{r^i_k}{r_k^j}\big)\bigg|+\bigg|\frac{x_k^i-x_k^j}{r_k^i}\bigg|\To\infty\quad \mbox{for } i\neq j,\, \, 1\leq i,\, j\leq m$ (see
\cite[Theorem 1.2]{PS-2}).
This complete the proof.
\end{proof}

\medskip

We end this section with the definition of some functions which will be used throughout the rest of the paper. We define,

\be\lab{17-7-4}J(u) := \frac{\|u\|_{\dot{H}^s(\Rn)}^2}{\Big(\displaystyle\int_{\R^N}a(x) |u(x)|^{2^*_s}{\rm d}x\Big)^{\tfrac{2}{2^*_s}}}, \quad
	J_{\infty}(u) := \frac{\|u\|_{\dot{H}^s(\Rn)}^2}{\Big(\displaystyle\int_{\R^N} |u(x)|^{2^*_s}{\rm d}x\Big)^{\tfrac{2}{2^*_s}}}.\ee
\be\lab{17-7-5}S := \inf_{u\in\dot{H}^s(\Rn)\setminus \{0\}} J_{\infty}(u),
	 \ee
i.e., $S$ is the best Sobolev constant. From \cite{CLO}, it is known that $S$ is achieved by the unique positive solution  (up to translation and dilation) $W$ of \eqref{AT0.8}. Further, as already noted in the above proof, $W$ is
	  radially symmetric positive decreasing smooth function satisfying \eqref{9-7-1} and
\begin{equation}\label{barI}\bar I_{1,0}(W)=\frac{s}{N}S^\frac{N}{2s}>0.\end{equation}

\medskip

\section{Proof of Theorem \ref{th:ex-f}}

In this section we prove Theorem \ref{th:ex-f}. To this aim we first establish existence of two positive critical points  in the spirit of \cite{BCG} for the following functional:
	\be\label{EF-1}
			I_{a,f}(u)=\frac{1}{2}\|u\|^{2}_{\dot{H}^s(\R^N)}-\frac{1}{2^*_s}\int_{\R^N}a(x)u_+^{2^*_s}\,{\rm d}x - \prescript{}{(\dot{H}^{s})'}{\langle}f,u{\rangle}_{\dot{H}^s},
		\ee
		where $u_+:=\max\{u, 0\}$ and $u_-:=-\min\{u, 0\}$ and $f\in \dot{H}^s(\R^N)'$ is a nonnegative nontrivial functional.
		
	Clearly, if $u$ is a critical points of $I_{a,f}$, then $u$ solves
\begin{equation}\label{MAT2}
\left\{\begin{aligned}
		(-\Delta)^s u &= a(x) u_+^{2^*_s-1}+f(x)\;\;\text{in}\;\mathbb{R}^{N},\\
		&u \in \dot{H}^{s}{(\mathbb{R}^{N})}.
		 \end{aligned}
  \right.
\end{equation}	
	
\begin{remark}\lab{r:30-7-3}
{\rm If $u$ is a weak solution of \eqref{MAT2} and $f$ is a nonnegative functional, then taking $v=u_-$ as a test function in \eqref{MAT2}, we obtain

$$-\|u_-\|^2_{\dot{H}^s(\Rn)} \, -\, \iint_{\R^{2N}} \frac{[u_+(y)u_-(x)+u_+(x)u_-(y)]}{|x - y|^{N + 2s}} \, {\rm d}x \, {\rm d}y
=\prescript{}{H^{-s}}{\langle}f,u_-{\rangle}_{H^s}\geq 0.$$ This in turn implies $u_-=0$, i.e., $u\geq 0$. Therefore, using maximum principle \cite[Theorem 1.2]{DPQ}, it follows that, $u$ is a positive solution to \eqref{MAT2}. Hence $u$ is a solution to \eqref{MAT1}.}
\end{remark}

To establish the existence of two critical points for $I_{a,f}$, we first need to prove some auxiliary results. Towards that, we partition $\Hs$ into three disjoint sets.
Let $g:\Hs\to\R$ be defined by $$g(u):=\|u\|_{\Hs}^2-(2^*_s-1) ||a||_{L^\infty(\Rn)}\|u\|^{2^*_s}_{L^{2^*_s}(\Rn)}.$$
Now, put
$$U_1:=\{u\in\Hs: u=0 \quad\text{or}\quad g(u)>0\}, \quad U_2:=\{u\in\Hs: g(u)<0\},$$
$$U:=\{u\in\Hs\setminus\{0\}: g(u)=0\}.$$

\begin{remark}\lab{r:30-7-1}
{\rm Using the Sobolev inequality, it is easy to see that $\|u\|_{\Hs}$ and $\|u\|_{L^{2^*_s}(\Rn)}$ are bounded away from $0$ for all $u\in U$.}
\end{remark}

Set
\be\lab{30-7-1}
c_0:=\inf_{U_1} {I_{a,f}}(u) \quad\text{and}\quad c_1:=\inf_{U} {I_{a,f}}(u).
\ee

\begin{remark}\lab{r:30-7-2}
{\rm  Clearly,  $g(tu)=t^2\|u\|_{\Hs}^2-t^{2^*_s}(2^*_s-1)\|a\|_{L^\infty(\Rn)}
\|u\|^{2^*_s}_{L^{2^*_s}(\Rn)}$ for any $t>0$ and $u\in\Hs$.
Moreover $g(0)=0$ and $t\mapsto g(tu)$ is a strictly concave function.
Thus, for any $u\in \Hs$, with $\|u\|_{\Hs}=1$, there exists unique $t=t(u)$ such that $tu\in U.$ On the other hand,    $g(tu)=(t^2-t^{2^*_s})\|u\|_{\Hs}^2$ for any $u\in U$. This implies
that
 $$tu\in U_1 \quad\text{for all}\quad t\in (0,1) \quad\text{and}\quad tu\in U_2 \quad\text{for all}\quad  t > 1.$$}
\end{remark}

\begin{lemma}\lab{l:1-8-1}
Assume that $C_0$ is defined as in Theorem \ref{th:ex-f}. Then
$$\frac{4s}{N+2s}\|u\|_{\Hs} \geq C_0S^{\tfrac{N}{4s}} \quad\mbox{for
all } u\in U,$$
where $S$ is defined in \eqref{17-7-5}.
\end{lemma}

\begin{proof} Note that
$$\|u\|_{L^{2^*_s}(\Rn)}=\frac{\|u\|_{\Hs}^\frac{2}{2^*_s}}
{\big((2^*_s-1)\|a\|_{L^\infty(\Rn)}\big)^\frac{1}{2^*_s}},\mbox{
whenever }u\in U.$$
Therefore, combining this with the definition of $S$, we have
$$\|u\|_{\Hs}\geq S^\frac{1}{2}\|u\|_{L^{2^*_s}(\Rn)} = S^\frac{1}{2}\frac{\|u\|_{\Hs}^\frac{2}{2^*_s}}{((2^*_s-1)
\|a\|_{L^\infty(\Rn)})^\frac{1}{2^*_s}}$$
for all $u\in U$. From here, using the definition of $C_0$, the lemma follows.
\end{proof}

\begin{lemma}\lab{l:30-7-1}
Assume $C_0$ is defined as in Theorem \ref{th:ex-f} and $c_0$ and $c_1$ are defined as in \eqref{30-7-1}. Further, if
\be\lab{J1.3}
\inf_{u\in\Hs,\, \|u\|_{L^{2^*_s}(\Rn)}=1}\bigg\{C_0\|u\|_{\Hs}^\frac{N+2s}{2s}
-\prescript{}{\hms}{\langle}f,u{\rangle}_{H^s}\bigg\}>0,
\ee
then $c_0<c_1$.
\end{lemma}

\begin{proof}
Define
\be\lab{31-7-1}\tilde J(u):=\frac{1}{2}\|u\|_{\Hs}^2-\frac{\|a\|_{L^\infty(\Rn)}}{2^*_s}
\|u\|^{2^*_s}_{L^{2^*_s}(\Rn)}-\prescript{}{\hms}{\langle}f,u{\rangle}_{\dot{H}^s}, \quad u\in\Hs.\ee

{\bf Step 1}: In this step we prove that there exists $\alpha>0$ such that
$$\frac{d}{dt}\tilde J(tu)|_{t=1}\geq \al \quad\mbox{for all } u\in U.$$
From the definition of $\tilde J$, we have $\frac{d}{dt}\tilde J(tu)|_{t=1}=\|u\|_{\Hs}^2-\|a\|_{L^\infty(\Rn)}\|u\|_{L^{2^*_s}(\Rn)}^{2^*_s}-\prescript{}{\hms}{\langle}f,u{\rangle}_{\dot{H}^s}$. Therefore, using the definition of $U$ and the value of $C_0$,  we have for  $u\in U$
\begin{align}\lab{30-7-4}
\frac{d}{dt}\tilde J(tu)|_{t=1}&=\frac{4s}{N+2s}\|u\|_{\Hs}^2-\prescript{}{\hms}{\langle}f,u{\rangle}_{\dot{H}^s}\no\\
&= \big((2^*_s-1)\|a\|_{L^\infty(\Rn)}\big)^\frac{N-2s}{4s}C_0\|u\|_{\Hs}^2-\prescript{}{\hms}{\langle}f,u{\rangle}_{\dot{H}^s}\\
&=C_0\frac{\|u\|^\frac{N+2s}{2s}_{\Hs}}{\|u\|_{L^{2^*_s}(\Rn)}^\frac{N}{2s}}-\prescript{}{\hms}{\langle}f,u{\rangle}_{\dot{H}^s}.\no
\end{align}
Further, $\eqref{J1.3}$ implies there  exists $d>0$ such that

		\be\lab{30-7-3}
		\inf_{u\in\Hs,\, \|u\|_{L^{2^*_s}(\Rn)}=1}\Big\{C_0\|u\|_{\Hs}^{\tfrac{N+2s}{2s}}-\prescript{}{\hms}{\langle}f,u{\rangle}_{\dot{H}^s}\Big\}\geq d.\ee
		Now,
		\Bea
		\eqref{30-7-3}&\Longleftrightarrow& C_0\frac{\|u\|_{\Hs}^{\tfrac{N+2s}{2s}}}{\|u\|_{L^{2^*_s}(\Rn)}^{\tfrac{N}{2s}}}-\prescript{}{\hms}{\langle}f,u{\rangle}_{\dot{H}^s} \geq d, \quad \|u\|_{L^{2^*_s}(\Rn)}=1\\
		&\Longleftrightarrow& C_0\frac{\|u\|_{\Hs}^{\tfrac{N+2s}{2s}}}{\|u\|_{L^{2^*_s}(\Rn)}^{\tfrac{N}{2s}}}-\prescript{}{\hms}{\langle}f,u{\rangle}_{\dot{H}^s} \geq d\|u\|_{L^{2^*_s}(\Rn)}, \quad u\in\Hs\setminus\{0\}.
\Eea
		Hence, plugging back the above estimate into \eqref{30-7-4} and using Remark \ref{r:30-7-1}, we complete the proof of Step 1.
		
		\vspace{2mm}
		
{\bf Step 2:} Let $(u_n)_n$ be a minimizing sequence for $I_{a,f}$ on $U$, i.e., $I_{a,f}(u_n)\to c_1$ and $\|u_n\|_{\Hs}^2=(2^*_s-1)\|a\|_{L^\infty(\Rn)}\|u_n\|_{L^{2^*_s}(\Rn)}^{2^*_s}$. Therefore, for large $n$
$$c_1+o(1)\geq I_{a,f}(u_n)\geq \tilde J(u_n)\geq \bigg(\frac{1}{2}-\frac{1}{2^*_s(2^*_s-1)}\bigg)\|u_n\|^2_{\Hs}- \|f\|_{\hms}\|u_n\|_{\Hs}.$$
This implies that $(\tilde J(u_n))_n$ is a bounded sequence and $(\|u_n\|_{\Hs})_n$ and
$(\|u_n\|_{L^{2^*_s}(\Rn)})_n$ are bounded.

{\bf Claim}: $c_0<0$.

Indeed, to prove the claim, it is enough to show that there exists $v\in U_1$ such that $I_{a,f}(v)<0$. Note that, thanks to Remark \ref{r:30-7-2}, we can choose $u\in U$ such that $\prescript{}{\hms}{\langle}f,u{\rangle}_{\dot{H}^s}>0$. Therefore,
$$I_{a,f}(tu)\leq t^2\bigg[\frac{(2^*_s-1)\|a\|_{L^\infty(\Rn)}}{2}-\frac{t^{2^*_s-2}}{2^*_s}\bigg]\int_{\Rn}|u|^{2^*_s}dx-t\prescript{}{\hms}{\langle}f,u{\rangle}_{\dot{H}^s} <0.$$
for $t<<1$. Moreover, $tu\in U_1$
 by Remark \ref{r:30-7-2}. Hence the claim follows.

Thanks to the above claim, $I_{a,f}(u_n)< 0$ for large $n$. Consequently,
$$0>I_{a,f}(u_n)\geq \bigg(\frac{1}{2}-\frac{1}{2^*_s(2^*_s-1)}\bigg)\|u_n\|^2_{\Hs}-\prescript{}{\hms}{\langle}f,u_n{\rangle}_{\dot{H}^s}.$$
This in turn implies $\prescript{}{\hms}{\langle}f,u_n{\rangle}_{\dot{H}^s}> 0$ for all large $n$.
Consequently, $\frac{d}{dt}\tilde J(tu_n)<0$ for $t>0$ small enough. Thus, by Step 1, there exists $t_n\in (0,1)$ such that $\frac{d}{dt}\tilde J(t_n u_n)=0$.
Moreover,  it is easy to check that for all $u\in U$, the function  $\frac{d}{dt}\tilde J(tu)$ is strictly increasing in $t\in [0,1)$ and therefore we can conclude that
 $t_n$ is unique.\medskip
	
{\bf Step 3}: In this step we show  that 	
\be\label{J3}
	\liminf_{n\rightarrow\infty}\{\tilde J(u_n)-\tilde J(t_nu_n)\}>0.
	\ee
	
Observe that $\tilde J(u_n)-\tilde J(t_nu_n)=\displaystyle\int_{t_n}^{1}\frac{d}{dt}\{\tilde  J(tu_n)\} \, {\rm d}t$ and that for all $n\in\mathbb{N}$ there is $\xi_n>0$ such that
	$t_n\in(0,\;1-2\xi_n)$ and $\frac{d}{dt}\tilde J(tu_n)\geq \al/2$ for $t\in[1-\xi_n,\;1]$.\\
	To establish \eqref{J3}, it is enough to show that $\xi_n>0$ can be chosen independent of $n\in\mathbb{N}$. But this is true since, $\frac{d}{dt}\tilde J(tu_n)|_{t=1}\geq \al$ and by the boundedness of $\{u_n\},$\\
	$$\bigg|\frac{d^2}{dt^2}\tilde J(tu_n)\bigg| = \bigg|\|u_n\|_{\Hs}^2 - (2^*_s-1)\|a\|_{L^\infty(\Rn)}t^{2^*_s-2}\int_{\Rn}|u_n|^{2^*_s}dx\bigg|= \bigg|(1-t^{2^*_s-2})\|u_n\|_{\Hs}^2\bigg|\leq C,$$
	for all $n\geq 1$ and $t\in[0,\;1]$.
	
\vspace{2mm}	
	
{\bf Step 4:} From the definition of $I_{a,f}$ and $\tilde J$, it immediately follows that $\frac{d}{dt}I_{a,f}(tu)\geq \frac{d}{dt}\tilde J(tu)$ for all $u\in\Hs$ and for all $t>0$. Hence,
$$I_{a,f}(u_n)-I_{a,f}(t_nu_n)= \int_{t_n}^{1}\frac{d}{dt}(I_{a,f}(tu_n))\; {\rm d}t\geq \int_{t_n}^{1}\frac{d}{dt}\tilde J(tu_n)\; {\rm d}t = \tilde J(u_n)-\tilde J(t_nu_n).$$
	Since $(u_n)_n\subset U$  is a minimizing sequence for $I_{a,f}$ on $U$, and $t_nu_n\in U_1,$ we conclude using \eqref{J3} that\\
	$$c_0 = \inf_{u\in U_1}I_{a,f}(u)< \inf_{u\in U}I_{a,f}(u)\equiv c_1$$
\end{proof}

Next, we introduce the equation at infinity associated to \eqref{MAT2}:
\be\lab{30-7-5}
(-\De)^s u=u_+^{2^*_s-1} \quad\text{in}\quad\Rn, \quad u\in\dot{H}^s(\Rn),
\ee
and the corresponding functional $I_{1,0}:\Hs\to\R$ defined by
$$I_{1,0}(u)=\frac{1}{2}\|u\|_{\Hs}^2-\frac{1}{2^*_s}\int_{\Rn}u_+^{2^*_s} {\rm d}x.$$
Arguing as in Remark \ref{r:30-7-3}, it immediately follows that solutions of \eqref{30-7-5} are the positive solutions of \eqref{AT0.8}.

\begin{proposition}\lab{p:30-7-1}
Assume that \eqref{J1.3} holds. Then $I_{a,f}$ has a critical point $u_0\in U_1$, with $I_{a,f}(u_0)=c_0$. In particular, $u_0$ is a positive weak solution to \eqref{MAT1}.
\end{proposition}

\begin{proof} We decompose the proof into few steps.

\vspace{2mm}

{\bf Step 1}: $c_0>-\infty$.

Note that $I_{a,f}(u)\geq \tilde J(u)$, where $\tilde J$ is defined as in \eqref{31-7-1}. Therefore, in order to prove Step 1, it is enough to show that $\tilde J$ is bounded from below. From definition of $U_1$, it immediately follows that
\be\lab{31-7-3}\tilde J(u)\geq \left[\frac{1}{2}-\frac{1}{2^*_s(2^*_s-1)}\right]\|u\|_{\Hs}^2 - \|f\|_{\hms}\|u\|_{\Hs}\;\, \mbox{for all}\, u\in {U}_1.\ee
As RHS is quadratic function in $\|u\|_{\Hs}$, $\tilde J$ is bounded from below. Hence Step 1 follows.

\vspace{2mm}

{\bf Step 2}: In this step we show that there exists a bounded  nonnegative $(PS)$ sequence
$(u_n)_n \subset U_1$ for $I_{a,f}$ at level $c_0$.

Let $(u_n)_n \subset \bar U_1$ such that $I_{a,f}(u_n)\to c_0$.
Since Lemma \ref{l:30-7-1} implies that $c_0<c_1$,  without restriction we can assume $(u_n)_n\subset U_1$. Further, using Ekeland's variational principle from $(u_n)_n $, we can extract a $(PS)$ sequence in $U_1$ for $I_{a,f}$ at level $c_0$. We again call it by $(u_n)_n $. Moreover, as $I_{a,f}(u)\geq \tilde J(u)$, from \eqref{31-7-3} it follows that $(u_n)_n $ is a bounded sequence. Therefore,  up to a subsequence $u_n\rightharpoonup u_0$ in $\Hs$ and $u_n\to u_0$ a.e. in $\Rn$. In particular, $ (u_n)_+\to (u_0)_+$ and $ (u_n)_-\to (u_0)_-$ a.e. in $\Rn$. Moreover, as $f$ is a nonnegative functional, we have
\Bea
o(1)&=&\prescript{}{\hms}\langle\bar I'_{a,f}(u_n), (u_n)_-\rangle_{\dot{H}^s}\\
&=&\langle u_n, (u_n)_-\rangle_{\Hs}-\int_{\Rn}a(x)(u_n)_+^{2^*_s-1}(u_n)_--\prescript{}{\hms}\langle\bar f, (u_n)_-\rangle_{\dot{H}^s}\\
&\leq&-\|(u_n)_-\|_{\Hs}^2-\iint_{\R^{2N}}\frac{(u_n)_-(x)(u_n)_+(y)+(u_n)_+(x)(u_n)_-(y)}{|x-y|^{N+2s}}dxdy\\
&\leq&-\|(u_n)_-\|_{\Hs}^2.
\Eea
Therefore, $(u_n)_-$ strongly converges to $0$ in $\Hs$ and so $(u_n)_-\to 0$ a.e. in $\Rn$ and also $(u_0)_-=0$ a.e. in $\Rn$. In other words, $u_0\geq 0$ a.e. in $\Rn$. Consequently, without loss of generality, we can assume that $(u_n)_n$ is a nonnegative sequence. This completes the proof of Step~2.

\vspace{2mm}

{\bf Step 3:}  In this step we show that $u_n\to u_0$ in $\Hs$ and $u_0\in U_1$.

Applying Proposition \ref{PSP}, we get
		 \be\label{J5}
		u_n-\bigg(u_0 +\sum_{j=1}^{m}a(x^j)^{-\frac{N-2s}{4s}}W^{r_n^j, x_n^j}\bigg) \To 0 \;\text{ in }\Hs.
\ee
	with $I_{a,f}'(u_0) =0$, $W$ is the unique positive solution of \eqref{AT0.8} and some appropriate sequences $(x_n^j)_n$,  $(r_n^j)_n$, with
either $x_n^j\to x^j$ or $|x_n^j|\to\infty$ and $r_n^j\to 0$. To prove Step~3, we need to show that $m=0$.
Arguing  by contradiction,
		suppose that $j\neq 0$ in \eqref{J5}. Then,
\bea\lab{12-4-3}
g\bigg(a(x^j)^{-\frac{N-2s}{4s}}W^{r_n^j, x_n^j}\bigg)&=&a(x^j)^{-\frac{N-2s}{2s}}\|W\|^2_{\Hs}-(2^*_s-1)\|a\|_{L^\infty(\Rn)}a(x^j)^{-\frac{N}{2s}}
\|W\|_{L^{2^*_s}(\R^N)}^{2^*_s}\no\\
&=&a(x^j)^{-\frac{N}{2s}}\bigg(a(x^j)-(2^*_s-1)\|a\|_{L^\infty(\Rn)}\bigg)\|W\|_{L^{2^*_s}(\R^N)}^{2^*_s}<0.
\eea				
From Proposition \ref{PSP}, we also have
		$$c_0=I_{a,f}(u_n)\rightarrow  I_{a,f}(u_0)+\sum_{j=1}^{m}a(x^j)^{-\frac{N-2s}{2s}}\bar I_{1,0}(W).$$
As $a>0$ and in force of~\eqref{barI}, from the above expression we obtain $ I_{a,f}(u_0)<c_0$. This in turn yields $u_0\not\in U_1$ and
\be\lab{12-4-4} g(u_0)\leq 0.\ee 	
	
Next, we evaluate $g\bigg(u_0 +\sum_{j=1}^{m}a(x^j)^{-\frac{N-2s}{4s}}\big(W^{r_n^j, x_n^j}\big)\bigg)$. Since $u_n\in U_1$, we have $g(u_n)\geq 0$. Therefore, the uniform continuity of $g$ and \eqref{J5}  give
		\be\label{J8}
		0\leq \liminf_{n\rightarrow\infty}g(u_n)=\liminf_{n\rightarrow\infty} g\bigg(u_0 +\sum_{j=1}^{m}a(x^j)^{-\frac{N-2s}{4s}}\big(W^{r_n^j, x_n^j}\big)\bigg).
\ee
We also note that if $u_0\neq 0$ then using Remark \ref{r:30-7-3}, we can say that $u_0$ is nonnegative. Therefore,
\bea\lab{12-4-1}
g\bigg(u_0 +\sum_{j=1}^{m}a(x^j)^{-\frac{N-2s}{4s}}\big(W^{r_n^j, x_n^j}\big)\bigg)&\leq&\|u_0\|_{\Hs}^2+\|\sum_{j=1}^{m}a(x^j)^{-\frac{N-2s}{4s}}\big(W^{r_n^j, x_n^j}\big)\|^2_{\Hs}\no\\
&&+2\big\langle u_0,\sum_{j=1}^{m}a(x^j)^{-\frac{N-2s}{4s}}\big(W^{r_n^j, x_n^j}\big) \big\rangle_{\Hs}\no\\
&&- (2^*_s-1)\|a\|_{L^\infty(\Rn)}\bigg(\int_{\Rn}|u_0|^{2^*_s}dx\no\\
&&\quad+\int_{\Rn}|\sum_{j=1}^{m}a(x^j)^{-\frac{N-2s}{4s}}\big(W^{r_n^j, x_n^j}\big)|^{2^*_s}dx\bigg)\no\\
&=&g(u_0)+g\bigg(\sum_{j=1}^{m}a(x^j)^{-\frac{N-2s}{4s}}\big(W^{r_n^j, x_n^j}\big)\bigg)\no\\
&&+2\big\langle u_0,\sum_{j=1}^{m}a(x^j)^{-\frac{N-2s}{4s}}\big(W^{r_n^j, x_n^j}\big) \big\rangle_{\Hs}
\eea
Similarly, it can be also shown that
\Bea
g\bigg(\sum_{j=1}^{m}a(x^j)^{-\frac{N-2s}{4s}}\big(W^{r_n^j, x_n^j}\big)\bigg)&\leq& \sum_{j=1}^m g\bigg(a(x^j)^{-\frac{N-2s}{4s}}\big(W^{r_n^j, x_n^j}\big)\bigg) \\
&&\quad+2\sum_{i,\, j=1}^{m}\bigg\langle a(x^i)^{-\frac{N-2s}{4s}}\big(W^{r_n^i, x_n^i}\big) ,a(x^j)^{-\frac{N-2s}{4s}}\big(W^{r_n^j, x_n^j}\big) \bigg\rangle_{\Hs}
\Eea
Substituting this into \eqref{12-4-1} yields
\bea\lab{12-4-2}
g\bigg(u_0 +\sum_{j=1}^{m}a(x^j)^{-\frac{N-2s}{4s}}\big(W^{r_n^j, x_n^j}\big)\bigg)&\leq&g(u_0)+\sum_{j=1}^m g\bigg(a(x^j)^{-\frac{N-2s}{4s}}\big(W^{r_n^j, x_n^j}\big)\bigg)\no\\
&&\quad+2\sum_{i,\, j=1}^{m}\bigg\langle a(x^i)^{-\frac{N-2s}{4s}}\big(W^{r_n^i, x_n^i}\big) ,a(x^j)^{-\frac{N-2s}{4s}}\big(W^{r_n^j, x_n^j}\big) \bigg\rangle_{\Hs}\no\\
&&\qquad+ 2\big\langle u_0,\sum_{j=1}^{m}a(x^j)^{-\frac{N-2s}{4s}}\big(W^{r_n^j, x_n^j}\big) \big\rangle_{\Hs}.
\eea
{\bf Claim:} $$(i)\,\, \bigg\langle u_0 ,\, a(x^j)^{-\frac{N-2s}{4s}}W^{r_n^j, x_n^j} \bigg\rangle_{\Hs}=o(1).$$
$$(ii)\, \, \bigg\langle a(x^i)^{-\frac{N-2s}{4s}}W^{r_n^i, x_n^i},\, a(x^j)^{-\frac{N-2s}{4s}}W^{r_n^j, x_n^j} \bigg\rangle_{\Hs}=o(1)$$
To prove (i), first we define $u_0^n(x):=(r_n^j)^\frac{N-2s}{2}u_0(x_n^j+r_n^jx)$. As $r_n^j\to 0$ and $u_0\in\dot{H}^s(\Rn)$, it is easy to see that
$u_0^n\deb 0$ in $\Hs$. Thus,
\Bea
&&\bigg\langle u_0 ,\, a(x^j)^{-\frac{N-2s}{4s}}W^{r_n^j, x_n^j} \bigg\rangle_{\Hs}\\
&=&a(x^j)^{-\frac{N-2s}{4s}}(r_n^j)^{-\frac{N-2s}{2}}\iint_{\R^{2N}}\frac{\big(u_0(x)-u_0(y)\big)\big(W(\frac{x-x_n^j}{r_n^j})-W(\frac{y-x_n^j}{r_n^j})\big)}{|x-y|^{N+2s}}dxdy\\
&=&a(x^j)^{-\frac{N-2s}{4s}}(r_n^j)^\frac{N-2s}{2}\iint_{\R^{2N}}\frac{\big(u_0(x_n^j+r_n^jx)-u_0(x_n^j+r_n^jy)\big)\big(W(x)-W(y)\big)}{|x-y|^{N+2s}}dxdy\\
&=&a(x^j)^{-\frac{N-2s}{4s}}\big\langle u_0^n ,W \big\rangle_{\Hs}\\
& =&o(1).
\Eea

Similarly,
\Bea
&&\bigg\langle a(x^i)^{-\frac{N-2s}{4s}}\big(W^{r_n^i, x_n^i}\big) ,a(x^j)^{-\frac{N-2s}{4s}}\big(W^{r_n^j, x_n^j}\big) \bigg\rangle_{\Hs}\\
&=&a(x^i)^{-\frac{N-2s}{4s}}a(x^j)^{-\frac{N-2s}{4s}}(r_n^i)^\frac{N-2s}{2}(r_n^j)^{-\frac{N-2s}{2}}\cdot\\
&&\qquad\times \iint_{\R^{2N}}\frac{\big(W(x)-W(y)\big)\big(W(\frac{r_n^ix+x_n^i-x_n^j}{r_n^j})-W(\frac{r_n^iy+x_n^i-x_n^j}{r_n^j})\big)}{|x-y|^{N+2s}}dxdy\\
&=&a(x^i)^{-\frac{N-2s}{4s}}a(x^j)^{-\frac{N-2s}{4s}}\big\langle W, W_n\big\rangle_{\Hs},
\Eea
where $W_n:=(\frac{r_n^i}{r_n^j})^\frac{N-2s}{2}W\big(\frac{r_n^i}{r_n^j}x+\frac{x_n^i-x_n^j}{r_n^j}\big)$. Further, we observe that  using the following
$$\bigg|\log\big(\frac{r^i_k}{r_k^j}\big)\bigg|+\bigg|\frac{x_k^i-x_k^j}{r_k^i}\bigg|\To\infty$$
from Proposition \ref{PSP}, it is easy to see that $W_n\deb 0$ in $\Hs$. Hence Claim (ii) follows.

A combination of the above claim along with \eqref{12-4-3} and \eqref{12-4-4} contradicts \eqref{J8}. Therefore, $j=0$ in \eqref{J5}. Hence, $u_n\to u_0$ in $\Hs$. Consequently, $g(u_n)\to g(u_0)$, which in turn implies $u_0\in \bar U_1$. But, since $c_0<c_1$, we can conclude $u_0\in U_1$.    Thus Step 3 follows.

		
\vspace{2mm}

{\bf Step 4:} From the previous steps we conclude that $I_{a,f}(u_0)=c_0$ and
$I_{a,f}'(u_0)=0$. Therefore, $u_0$ is a weak solution to \eqref{MAT2}. Combining this with Remark \ref{r:30-7-3}, we conclude the proof of the proposition.
\end{proof}

\begin{proposition}\lab{p:31-7-2}
Assume \eqref{J1.3} holds and either $a\equiv 1$ or $\|a\|_{L^\infty(\Rn)}\geq \al(N,s)$, where $\al(N,s)$ is the second zero of the function \eqref{12-8-3}. Then $I_{a,f}$ has a second critical point $v_0\neq u_0$. In particular, $v_0$ is a positive solution to \eqref{MAT1}.
\end{proposition}

\begin{proof}
Let $u_0$ be the critical point obtained in Proposition \ref{p:30-7-1} and $W$ be the unique
positive solution of \eqref{AT0.8}. Set, $w_t(x):=W\big(\frac{x}{t}\big)$ and let $\bar x_0\in\Rn$ such that $a(\bar x_0)=\|a\|_{L^\infty(\Rn)}$.

{\bf Claim 1:} $u_0+a(\bar x_0)^{-\frac{N-2s}{4s}}w_t\in U_2$ for $t>0$ large enough.

Indeed, as $\|a\|_{L^\infty(\Rn)}\geq 1$ and $u_0,\, w_t>0$, using Young inequality with $\eps>0$, we obtain
\Bea
g(u_0+a(\bar x_0)^{-\frac{N-2s}{4s}}w_t)&\leq&\|u_0\|^2_{\Hs}+a(\bar x_0)^{-\frac{N-2s}{2s}}\|w_t\|^2_{\Hs}+2a(\bar x_0)^{-\frac{N-2s}{4s}}\<u_0,w_t\>_{\Hs} \\
&&\quad-(2^*_s-1)(\|u_0\|^{2^*_s}_{L^{2^*_s}(\Rn)}+a(\bar x_0)^{-\frac{N}{2s}}\|w_t\|^{2^*_s}_{L^{2^*_s}(\Rn)})\\
&\leq&(1+\eps)a(\bar x_0)^{-\frac{N-2s}{2s}}\|w_t\|^2_{\Hs}+(1+C(\eps))\|u_0\|^2_{\Hs}\\
&&\quad-(2^*_s-1)(\|u_0\|^{2^*_s}_{L^{2^*_s}(\Rn)}+a(\bar x_0)^{-\frac{N}{2s}}\|w_t\|^{2^*_s}_{L^{2^*_s}(\Rn)})\\
&=&(1+C(\eps))\|u_0\|^2_{\Hs}-(2^*_s-1) \|u_0\|^{2^*_s}_{L^{2^*_s}(\Rn)}\\
&&\quad+\|W\|_{\Hs}^2\big[(1+\eps)a(\bar x_0)^{-\frac{N-2s}{2s}}t^{N-2s}-(2^*_s-1)a(\bar x_0)^{-\frac{N}{2s}}t^N\big]\\
&\leq& 0 \quad\mbox{for } t>0 \mbox{  large enough}.
\Eea
Therefore, $g(u_0+a(\bar x_0)^{-\frac{N-2s}{4s}}w_t)<0$ for $t$  large enough. Hence the claim follows.

\vspace{2mm}

{\bf Claim 2:} $I_{a,f}\bigg(u_0+a(\bar x_0)^{-\frac{N-2s}{4s}}w_t\bigg)<I_{a,f}(u_0)+I_{1,0}\bigg(a(\bar x_0)^{-\frac{N-2s}{4s}}w_t\bigg),\;\forall\, t>0 $.

Indeed, since $u_0, \, w_t>0$, taking $a(\bar x_0)^{-\frac{N-2s}{4s}}w_t$ as the test function for \eqref{MAT2} yields $$ {\langle} u_0,\;a(\bar x_0)^{-\frac{N-2s}{4s}}w_t{\rangle}_{\Hs} =\int_{\Rn}a(x)a(\bar x_0)^{-\frac{N-2s}{4s}}u_0^{2^*_s-1}w_t\, {\rm d}x+a(\bar x_0)^{-\frac{N-2s}{4s}}\prescript{}{\hms}{\langle}f,w_t{\rangle}_{\dot{H}^s}.$$
Consequently, using the above expression and the fact that $a\geq 1$, we obtain
\Bea
I_{a,f}\big(u_0+a(\bar x_0)^{-\frac{N-2s}{4s}}w_t\big)
&=& \frac{1}{2}\|u_0\|_{\Hs}^2+\frac{a(\bar x_0)^{-\frac{N-2s}{2s}}}{2}\|w_t\|_{\Hs}^2+ a(\bar x_0)^{-\frac{N-2s}{4s}}{\langle} u_0,\;w_t{\rangle}_{\Hs} \\
&&\qquad-\frac{1}{2^*_s}\int_{\Rn}a(x)(u_0+a(\bar x_0)^{-\frac{N-2s}{4s}}w_t)^{2^*_s}\;{\rm d}x-\prescript{}{\hms}{\langle}f,u_0{\rangle}_{\dot{H}^s}\\
&&\qquad\qquad-a(\bar x_0)^{-\frac{N-2s}{4s}}\prescript{}{\hms}{\langle}f,w_t{\rangle}_{\dot{H}^s}\\
&=&I_{a,f}(u_0)+I_{1,0}\big(a(\bar x_0)^{-\frac{N-2s}{4s}}w_t\big) +a(\bar x_0)^{-\frac{N-2s}{4s}}\langle u_0,\;w_t\rangle_{\Hs}\\
&&\quad+\frac{1}{2^*_s}\int_{\Rn}a(x)u_0^{2^*_s}\;{\rm d}x
+\frac{a(\bar x_0)^{-\frac{N}{2s}}}{2^*_s}\int_{\Rn}w_t^{2^*_s}\;{\rm d}x\\
&&\quad-\frac{1}{2^*_s}\int_{\Rn}a(x)(u_0+a(\bar x_0)^{-\frac{N-2s}{4s}}w_t)^{2^*_s}\;{\rm d}x-a(\bar x_0)^{-\frac{N-2s}{4s}}\prescript{}{\hms}{\langle}f,w_t{\rangle}_{\dot{H}^s}\\
&\leq&I_{a,f}(u_0)+I_{1,0}\big(a(\bar x_0)^{-\frac{N-2s}{4s}}w_t\big) +\frac{1}{2^*_s}\int_{\Rn}a(x)\bigg[2^*_sa(\bar x_0)^{-\frac{N-2s}{4s}}u_0^{2^*_s-1}w_t\\
&&\qquad\qquad+|u_0|^{2^*_s}+a(\bar x_0)^{-\frac{N}{2s}}w_t^{2^*_s}-(u_0+a(\bar x_0)^{-\frac{N-2s}{4s}}w_t)^{2^*_s}\bigg]{\rm d}x\\
&<&I_{a,f}(u_0)+I_{1,0}\big(a(\bar x_0)^{-\frac{N-2s}{4s}}w_t\big).
\Eea
Hence the Claim follows.

\vspace{2mm}

A direct computation shows that
\be\lab{1-8-2}
\lim_{t\to\infty}I_{1,0}\big(a(\bar x_0)^{-\frac{N-2s}{4s}}w_t\big)=-\infty
\ee
From \eqref{1-8-2} and using the relation
$$\|w_t\|_{\Hs}^2=t^{N-2s}\|W\|_{\Hs}^2, \quad \|w_t\|_{L^{2^*_s}(\Rn)}^{2^*_s}=t^{N}\|W\|_{\Hs}^2,$$ a straight forward computation yields that
$$
\sup_{t>0} I_{1,0}\big(a(\bar x_0)^{-\frac{N-2s}{4s}}w_t\big)=I_{1,0}\big(a(\bar x_0)^{-\frac{N-2s}{4s}}w_{t_{max}}\big), \quad\mbox{where } t_{max}=a(\bar x_0)^{-\frac{1}{2s}}.
$$
Therefore, substituting the value of $t_{max}$ in the definition of $I_{a,f}$, it is not difficult to check that
$$\sup_{t>0} I_{1,0}\big(a(\bar x_0)^{-\frac{N-2s}{4s}}w_t\big)=a(\bar x_0)^{-\frac{N}{s}}\bigg(\frac{a(\bar x_0)^2}{2}-\frac{1}{2^*_s}\bigg)\|W\|^2_{\Hs}.$$
Combining this with Claim 2 and \eqref{1-8-2} yields
\be\begin{split}\lab{1-8-6}
I_{a,f}(u_0+a(\bar x_0)^{-\frac{N-2s}{4s}}w_t)<I_{a,f}(u_0)+a(\bar x_0)^{-\frac{N}{s}}\bigg(\frac{a(\bar x_0)^2}{2}-\frac{1}{2^*_s}\bigg)\|W\|^2_{\Hs}  \quad\forall\, t>0 \\
\mbox{ and }\qquad I_{a,f}(u_0+a(\bar x_0)^{-\frac{N-2s}{4s}}w_t)<I_{a,f}(u_0), \quad\mbox{for } t \mbox{ large enough}.
\end{split}
\ee
 Fix $t_0>0$ large enough such that \eqref{1-8-6} and Claim 1 are satisfied.
Next, we set
$$\ga:=\inf_{i\in\Ga}\max_{t\in[0,1]} I_{a,f}\big(i(t)\big),$$
where $$\Ga:=\{i\in C\big([0,1], \Hs\big) : i(0)=u_0,\quad i(1)= u_0+a(\bar x_0)^{-\frac{N-2s}{4s}}w_{t_0}\}.$$
As $u_0\in U_1$ and $u_0+a(\bar x_0)^{-\frac{N-2s}{4s}}w_{t_0}\in U_2$, for every $i\in \Ga$, there exists $t_i\in(0,1)$
such that $i(t_i)\in U$. Therefore, $$\max_{t\in[0,1]} I_{a,f}(i(t))\geq I_{a,f}\big(i(t_i)\big)\geq \inf_{U}I_{a,f}(u)=c_1.$$
Thus, $\ga\geq c_1>c_0=I_{a,f}(u_0)$.
Here in the last inequality we have used Lemma \ref{l:30-7-1}.

\vspace{2mm}

{\bf Claim 3:} $\ga<I_{a,f}(u_0)+a(\bar x_0)^{-\frac{N}{s}}\bigg(\frac{a(\bar x_0)^2}{2}-\frac{1}{2^*_s}\bigg)\|W\|^2_{\Hs}$.

It is easy to see that  $\lim_{t\to 0}\|w_t\|_{\Hs}=0$. Thus, if we define
$\tilde i(t)=u_0+a(\bar x_0)^{-\frac{N-2s}{4s}}w_{tt_0}$, then $\lim_{t\to 0}\|\tilde i(t)-u_0\|_{\Hs}=0$. Consequently, $\tilde i\in \Ga$. Therefore, using \eqref{1-8-6}, we obtain
$$\ga\leq \max_{t\in[0,1]}I_{a,f}(\tilde i(t))=\max_{t\in[0,1]}I_{a,f}(u_0+a(\bar x_0)^{-\frac{N-2s}{4s}}w_{tt_0})<I_{a,f}(u_0)+a(\bar x_0)^{-\frac{N}{s}}\bigg(\frac{a(\bar x_0)^2}{2}-\frac{1}{2^*_s}\bigg)\|W\|^2_{\Hs}.$$
Thus the claim follows.

Hence \be\lab{12-8-2}I_{a,f}(u_0)<\ga<I_{a,f}(u_0)+a(\bar x_0)^{-\frac{N}{s}}\bigg(\frac{a(\bar x_0)^2}{2}-\frac{1}{2^*_s}\bigg)\|W\|^2_{\Hs}.\ee

{\bf Claim 4:} $a(\bar x_0)^{-\frac{N}{s}}\bigg(\frac{a(\bar x_0)^2}{2}-\frac{1}{2^*_s}\bigg)\|W\|^2_{\Hs}\leq a(\bar x_0)^{-\frac{N-2s}{2s}}I_{1,0}(W)$. 

Since $I_{1,0}(W)=\frac{s}{N}\|W\|^2_{\Hs}$, we observe that
\bea\lab{12-8-1}
&&a(\bar x_0)^{-\frac{N}{s}}\bigg(\frac{a(\bar x_0)^2}{2}-\frac{1}{2^*_s}\bigg)\|W\|^2_{\Hs}\leq a(\bar x_0)^{-\frac{N-2s}{2s}}I_{1,0}(W)\no\\
&&\text{if and only if}\,\, \frac{s}{N}a(\bar x_0)^\frac{N+2s}{2s}-\frac{a(\bar x_0)^2}{2}+\frac{1}{2^*_s}\geq 0.
\eea
Define $$\varphi(t):=\frac{s}{N}t^\frac{N+2s}{2s}-\frac{t^2}{2}+\frac{1}{2^*_s}.$$
Then an easy analysis shows that $\varphi(1)=0$ and there exists $\al(N,s)>1$ such that $\varphi(t)>0$ for all $t>\al(N,s)$, $\varphi(t)<0$ for $t\in(1, \al(N,s))$.

Therefore, if $a(\bar x_0)=1$ (which is equivalent to $a\equiv 1$) or $\|a\|_{L^\infty(\Rn)}\geq \al(N,s)$ then \eqref{12-8-1} holds. Hence, using the hypothesis of Proposition \ref{p:31-7-2}, we have
$$a(\bar x_0)^{-\frac{N}{s}}\bigg(\frac{a(\bar x_0)^2}{2}-\frac{1}{2^*_s}\bigg)\|W\|^2_{\Hs}\leq a(\bar x_0)^{-\frac{N-2s}{2s}}I_{1,0}(W)\leq a(x)^{-\frac{N-2s}{2s}}I_{1,0}(W),$$
for all $x\in\Rn$. Substituting this into \eqref{12-8-2}, yields
$$I_{a,f}(u_0)<\ga<I_{a,f}(u_0)+a(x)^{-\frac{N-2s}{2s}}I_{1,0}(W) \quad\mbox{for all } x\in\Rn.$$
Using Ekeland's variational principle, there exists a $(PS)$ sequence $(u_n )_n$
for $I_{a,f}$ at level $\ga$. Doing a standard computation yields $(u_n)_n$ is bounded sequence. Further as, $\ga<I_{a,f}(u_0)+a(x)^{-\frac{N-2s}{2s}}I_{1,0}(W)$ (for any $x\in\Rn$),  from Proposition \ref{PSP} we can conclude that $u_n\to v_0$, for some $v_0\in \Hs$ such that $I_{a,f}'(v_0)=0$ and $I_{a,f}(v_0)=\ga$. Further, as
$I_{a,f}(u_0)<\ga$, we conclude $v_0\neq u_0$.

$I_{a,f}'(v_0)=0\Longrightarrow v_0$ is a weak solution to \eqref{MAT2}. Combining this with Remark~\ref{r:30-7-3}, we conclude the proof of the proposition.
\end{proof}

\begin{lemma}\lab{l:J1.3}
If $\|f\|_{\hms}<C_0S^\frac{N}{4s}$, then \eqref{J1.3} holds.
\end{lemma}
\begin{proof}
Using the given hypothesis, we can obtain $\eps>0$ such that $\|f\|_{\hms}<C_0S^\frac{N}{4s}-\eps$. Therefore, using Lemma \ref{l:1-8-1}, we have
$$ \prescript{}{\hms}{\langle}f, u{\rangle}_{\dot{H}^s} \leq \|f\|_{\hms}\|u\|_{\Hs}<\big[C_0S^\frac{N}{4s}-\eps\big]\|u\|_{\Hs}\leq \frac{4s}{N+2s}\|u\|_{\Hs}^2-\eps\|u\|_{\Hs},$$
for all $u\in U$. Therefore,
$$\frac{4s}{N+2s}\|u\|_{\Hs}^2-\prescript{}{\hms}{\langle}f,u{\rangle}_{\dot{H}^s}> \eps\|u\|_{\Hs}$$
for all  $u\in U$, i.e., $$\inf_{U}\bigg[\frac{4s}{N+2s}\|u\|_{\Hs}^2-\prescript{}{\hms}{\langle}f,u{\rangle}_{\dot{H}^s}\bigg]\geq \eps\inf_{U}\|u\|_{\Hs}.$$
Since, by Remark \ref{r:30-7-1}, we have $\|u\|_{\Hs}$ is bounded away from $0$ on $U$,  the above expression implies
\be\lab{1-8-1}\inf_{U}\bigg[\frac{4s}{N+2s}\|u\|_{\Hs}^2-\prescript{}{\hms}{\langle}f,u{\rangle}_{\dot{H}^s}\bigg]>0.\ee
On the other hand,
\bea\lab{1-8-2'}
\eqref{J1.3}&\iff& C_0\frac{\|u\|_{\Hs}^{\tfrac{N+2s}{2s}}}{\|u\|_{L^{2^*}(\Rn)}^{\tfrac{N}{2s}}} - \prescript{}{\hms}{\langle}f,u{\rangle}_{\dot{H}^s} >0\quad \text{for}\quad \|u\|_{L^{2^*_s}(\Rn)}=1\no\\
		&\iff& C_0\frac{\|u\|_{\Hs}^{\tfrac{N+2s}{2s}}}{\|u\|_{L^{2^*}(\Rn)}^{\tfrac{N}{2s}}} - \prescript{}{\hms}{\langle}f,u{\rangle}_{\dot{H}^s}  >0\quad \text{for}\quad u\in U\no\\
		&\iff& \frac{4s}{N+2s}\|u\|_{\Hs}^2 -\prescript{}{H^{-s}}{\langle}f,u{\rangle}_{H^s}> 0 \quad\text{for }\quad u\in U.
		\eea
Clearly, \eqref{1-8-1} insures RHS of \eqref{1-8-2'} holds. Hence the lemma follows.
\end{proof}

\vspace{2mm}

{\bf Proof of Theorem \ref{th:ex-f} completed:}
\begin{proof}
Combining Proposition \ref{p:30-7-1} and Proposition \ref{p:31-7-2} with Lemma \ref{l:J1.3}, we conclude the proof of Theorem \ref{th:ex-f}.
\end{proof}

\section{Proof of Theorem \ref{th:2}}

\begin{proof}
We observe that
	\begin{equation}\label{esti-1}
	I''_{a,f}(u)(h,h) =  \|h\|_{\Hs}^2 -(2^*_s-1)\int_{\Rn} a(x)u_+^{2^*_s-2}h^2 \, {\rm d}x
	\end{equation}
Since $a\in L^\infty(\Rn)$, using H\"{o}lder inequality and Sobolev inequality, we estimate the second term on the RHS as follows
	
	\begin{eqnarray*}
	\int_{\Rn} a(x)u_+^{2^*_s-2}h^2 \, {\rm d}x &\leq& \|a\|_{L^\infty(\Rn)}\left( \int_{\Rn} |u|^{2^*_s} \, {\rm d}x\right)^{\frac{2s}{N}}
	\left(\int_{\Rn} |h|^{2^*_s} \, {\rm d}x \right)^{\frac{2}{2^*_s}}\\
	&\leq& \|a\|_{L^\infty(\Rn)}S^{-\frac{2^*_s}{2}} \, \|u\|^{2^*_s-2}_{\Hs} \, \|h\|^{2}_{\Hs}.
	\end{eqnarray*}		
	Thus substituting the above in \eqref{esti-1} we obtain
	
	$$ I''_{a,f}(u)(h,h) \geq \left( 1 - (2^*_s-1)\|a\|_{L^\infty(\Rn)}S^{-\frac{2^*_s}{2}} \|u\|^{2^*_s-2}_{\Hs}\right) \|h\|^{2}_{\Hs}.$$	
Therefore, 		
	$I''_{a,f}(u)$ is positive definite for $u\in B(r_1)$, with
	$r_1 =\big((2^*_s-1)\|a\|_{L^\infty(\Rn)}\big)^{-\frac{1}{2^*_s-2}}S^\frac{N}{4s}$ and hence $I_{a,f}$ is strictly convex in $B(r_1).$
	
	\medskip
	
	 For $\|u\|_{\Hs} = r_1$,
	\bea
	\Iaf &=& \tfrac{1}{2}\|u\|_{\Hs}^2 - \tfrac{1}{2^*_s}\int_{\Rn}a(x)u_+^{2^*_s}\; {\rm d}x-  \prescript{}{\hms}{\langle}f, u{\rangle}_{\dot{H}^s}\no\\
	&\geq&\bigg(\frac{1}{2}-\frac{1}{2^*_s}\|a\|_{L^\infty(\Rn)} S^{-\frac{2^*_s}{2}}  r_1^{2^*_s-2}\bigg) r_1^2 - r_1\|f\|_{\hms}\no
	\eea
Since, $r_1^{2^*_s-2} \, = \, \frac{1}{(2^*_s-1)\|a\|_{L^\infty(\Rn)}} S^{\frac{2^*_s}{2}}$, we obtain	
	$$
	\Iaf \geq \bigg(\frac{1}{2} -\frac{1}{2^*_s(2^*_s-1)}\bigg)r_1^2-r_1\|f\|_{\hms}.
	$$
Thus there exists $d>0$ such that
	$$
	\inf_{\|u\|_{\Hs} \, = \, r_1}\Iaf >0, \quad  \mbox{provided that } \  0<\|f\|_{\hms}\leq d.
	$$
Since $I_{a,f}$ is strictly convex in $B(r_1)$ and $\inf_{\|u\|_{\Hs} \, = \, r_1}\Iaf  >0=I_{a,f}(0)$,  there exists a unique critical point $u_0$ of $I_{a,f}$ in $B(r_1)$ and it satisfies
	\be\lab{2-9-1}
	I_{a,f}(u_0) = \inf_{\|u\|_{\Hs}<r_1}\Iaf<I_{a,f}(0)=0,
	\ee
where the last inequality is due to {the}
strict convexity of $I_{a,f}$ in $B(r_1)$.	
	Combining this with Remark~\ref{r:30-7-3}, we conclude the proof of the theorem.
\end{proof}

\appendix
\section{Morrey space}
\setcounter{equation}{0}

We recall the definition of the homogeneous Morrey spaces $L^{r,\ga}(\Rn)$ , introduced by Morrey as a
refinement homogeneous of the usual Lebesgue spaces. A measurable function $u:\Rn\to\R$
belongs to the Morrey space $L^{r,\ga}(\Rn)$, with $r\in[1,\infty)$ and $\ga\in[0,N]$ if and only if
\be\lab{12-13-1}
\|u\|_{L^{r,\ga}(\Rn)}^r:=\sup_{R>0,\, x\in\Rn} R^\ga\fint_{B(x,R)}|u|^r dy.
\ee
From the above definition, it is clear that if $\ga=N$ then $L^{r,N}(\Rn)$ coincided with usual Lebesgue space $L^r(\Rn)$ for any $r\geq 1$ and similarly $L^{r,0}(\Rn)$ coincides with $L^\infty(\Rn)$. It is interesting to note that $L^{r,\ga}$ experiences same translation and dilation invariance as in $L^{2^*_s}$ and therefore of $\Hs$ if $\frac{\ga}{r}=\frac{N-2s}{2}$. Let $(u)^{x_0,r}$ be the function defined by \eqref{12-13-3}. By change of variable
formula, one can see that the following equality holds
$$\|(u)^{x_0,r}\|_{L^{r, \frac{N-2s}{2}r}}=\|u\|_{L^{r, \frac{N-2s}{2}r}(\Rn)},$$
for any $r\in[1, 2^*_s]$. Next, we recall a result from \cite{PS-1}.

\begin{lemma}\lab{l:12-13-1}\cite[Theorem 1]{PS-1}
For any $0<s<N/2$  there exists a constant $C$ depending only on $N$ and $s$ such that, for any $2/2^*_s\leq \theta<1$ and for any $1\leq r<2^*_s$
$$\|u\|_{L^{2^*_s}(\Rn)}\leq C\|u\|_{\Hs}^\theta\|u\|^{1-\theta}_{L^{r, r(N-2s)/2}(\Rn)} \quad\mbox{for all } u\in\Hs.$$
\end{lemma}

Note that using the H\"{o}lder inequality, we also have
$L^{2^*_s}(\Rn)\hookrightarrow L^{r, \frac{N-2s}{2}r}(\Rn)$
is continuous, i.e., there
exists a constant $C=C(N,s)$ such that
\be\lab{12-13-4}\|u\|_{L^{r, r(N-2s)/2}(\Rn)}\leq C\|u\|_{L^{2^*_s}(\Rn)} \quad\mbox{for all } u\in L^{2^*_s}(\Rn).\ee
For more details about the Morrey spaces, we refer to \cite{PS-1}.

\medskip

{\bf Acknowledgement}:  M. Bhakta wishes to express her sincere gratitude to the Dipartimento di Matematica e Informatica of
the Universit\`{a} degli Studi di Perugia, where part of this work started during a visit of her to that institution.
The research of M.~Bhakta is partially supported by the {\em SERB MATRICS grant (MTR/2017/000168)}.

P. Pucci was partly supported by the Italian MIUR project
{\em Variational methods, with applications to problems in mathematical physics and
geometry} (2015KB9WPT\_009) and is member of the {\em Gruppo Nazionale per
l'Analisi Ma\-te\-ma\-ti\-ca, la Probabilit\`a e le loro Applicazioni}
(GNAMPA) of the {\em Istituto Nazionale di Alta Matematica} (INdAM).
P. Pucci was also partly supported by of the {\em Fondo Ricerca di Base di Ateneo --
Eser\-ci\-zio 2017--2019} of the University of Perugia, named {\em PDEs and Nonlinear Analysis}.

\end{document}